\documentclass[twocolumn]{revtex4}
\usepackage{latexsym}
\usepackage{amsmath}
\usepackage{times}
\usepackage{amssymb}
\usepackage{fancyheadings}
\usepackage[T1]{fontenc}
\usepackage[utf8]{inputenc}
\usepackage{fancyhdr}
\usepackage{url}
\usepackage{hyperref}
\usepackage{color}

\usepackage{multirow}
\usepackage{enumerate}
\usepackage{verbatim}
\usepackage{tikz}
\usetikzlibrary{matrix,arrows}
\usepackage[caption=false]{subfig}
\usepackage[ruled]{algorithm2e}
\usepackage[export]{adjustbox}
\usepackage{etoolbox}
\robustify{\subref}

\usepackage{graphicx}

\graphicspath{{./}}

\allowdisplaybreaks

\def \vec#1{{\bf{#1}}}

\makeatletter
\def\hlinewd#1{%
\noalign{\ifnum0=`}\fi\hrule \@height #1 %
\futurelet\reserved@a\@xhline}
\makeatother

\newcommand{\bi}{\begin{itemize}}
\newcommand{\ei}{\end{itemize}}

\newcommand{\diverg}{\vec{\nabla}\cdot}
\newcommand{\director}{\vec{n}}
\newcommand{\curl}{\vec{\nabla}\times}
\newcommand{\Ltwoinner}[3]{\langle #1,#2 \rangle_0}
\newcommand{\Ltwonorm}[2]{\Vert #1 \Vert_0}

\newcommand{\Ltwoinnerndim}[4]{\langle #1,#2 \rangle_0}

\newcommand{\Rone}{\mathbb{R}}

\newcommand{\diff}[1]{\, d#1}

\newcommand{\ltwoinner}[2]{( #1, #2 )}

\newcommand{\Honenorm}[2]{\Vert #1 \Vert_1}

\newcommand{\Hone}[1]{H^1(#1)}

\newcommand{\Honenot}[1]{H^1_0({#1})}

\newcommand{\Honeb}[1]{H^{1}_g(#1)}

\newcommand{\Ltwo}[1]{L^2(#1)}

\newcommand{\Lp}[1]{L^p (\Omega)}

\newcommand{\Linfinity}[1]{L^{\infty}(\Omega)}

\newcommand{\triangulation}{\mathcal{T}_h}

\newcommand{\diam}{\text{diam }}

\def\@begintheorem#1#2{\par\bgroup{\scshape #1\ #2. }\it\ignorespaces}
\def\@opargbegintheorem#1#2#3{\par\bgroup%
   {\scshape #1\ #2\ ({\upshape #3}). }\it\ignorespaces}
\def\@endtheorem{\egroup}

\if@onethmnum
  \newtheorem{theorem}{Theorem}
  \newtheorem{lemma}[theorem]{Lemma}
  \newtheorem{corollary}[theorem]{Corollary}
  \newtheorem{proposition}[theorem]{Proposition}
  \newtheorem{definition}[theorem]{Definition}
\else
  \newtheorem{theorem}{Theorem}[section]
  \newtheorem{lemma}[theorem]{Lemma}
  \newtheorem{corollary}[theorem]{Corollary}
  \newtheorem{proposition}[theorem]{Proposition}
  
\fi

\begin{document}


\title{\bf A Posteriori Error Estimators for the Frank-Oseen Model of Liquid Crystals}
\author{D. B. Emerson$^{\dag,}$\vspace*{3mm}}\thanks{Corresponding Author: (david.emerson@tufts.edu).}
\affiliation{$^\dag$Department of Mathematics, Tufts University, Medford, MA, United States 02155 }
\date{\today}


\begin{abstract}
\noindent This paper derives a posteriori error estimators for the nonlinear first-order optimality conditions associated with the Frank-Oseen elastic free-energy model of nematic and cholesteric liquid crystals, where the required unit-length constraint is imposed via either a Lagrange multiplier or penalty method. Furthermore, theory establishing the reliability of the proposed error estimator for the penalty method is presented, yielding a concrete upper bound on the approximation error of discrete solutions. The error estimators herein are composed of readily computable quantities on each element of a finite-element mesh, allowing the formulation of an efficient adaptive mesh refinement strategy. Four elastic equilibrium problems are considered to examine the performance of the error estimators and corresponding adaptive mesh refinements against that of a simple uniform refinement scheme. The adapted grids successfully provide significant reductions in computational work while producing solutions that are highly competitive with those of uniform mesh in terms of constraint conformance and computed free energies.

\vspace*{2ex}\noindent\textit{\bf Keywords}: liquid crystal simulation, a posteriori error estimators, adaptive mesh refinement, nested iteration.
\\[3pt]
\noindent\textit{\bf AMS}: 76A15, 65N30, 49M15, 65N22, 65N15
\end{abstract}

\maketitle

\thispagestyle{fancy}

\section{Introduction}

Liquid crystals are used and studied in a diverse array of modern applications including display technologies, nanoparticle organization \cite{Lagerwall1}, and liquid crystal infused elastomers used in the production of novel actuator devices such as artificial muscles \cite{Thomsen1} and light-driven motors \cite{Yamada1}, among many others. They are a form of soft material that exhibits mesophases, depending on temperature, with properties intermediate between liquid and crystalline phases. In this paper, we focus on nematic and cholesteric liquid crystals which consist of rod-like molecules with a preferred average orientation at each point denoted by the vector field $\director(x, y, z) = (n_1, n_2, n_3)^T$. In the model considered below, this vector field is subject to a pointwise unit-length constraint throughout a given domain, $\Omega$. Thorough treatments of liquid crystal physics are found in \cite{DeGennes1, Stewart1, Virga1}.

Numerical studies of liquid crystal behavior are a fundamental component in the validation and analysis of experiments, exploration of novel physical phenomena \cite{Emerson2, Atherton1}, and investigation of device design and performance. Many current experiments and technologies require simulations with anisotropic physical constants and intricate boundary conditions on two and three dimensional domains. This paper focuses on liquid crystal simulations performed by solving the first-order optimality conditions derived from the Frank-Oseen elastic free-energy model. As seen in \cite{Emerson1, Emerson2, Emerson3}, this approach yields an effective method for simulating complicated physical phenomena, including flexoelectric effects. However, the variational system resulting from the derivation of the first-order optimality conditions is highly nonlinear. Coupled with the nonlinear pointwise unit-length constraint and the desire to accurately simulate behavior on higher dimensional domains, this motivates the development of an a posteriori error estimator for numerical solutions to the optimality conditions. In the following, error estimators are derived for the optimality conditions arising when the unit-length constraint is imposed with either a Lagrange multiplier or a penalty term.

A posteriori error estimators aim to provide easily computable and reliable bounds on the error of numerical solutions for partial differential equations (PDEs) and variational systems. Accurate error estimators significantly increase the efficiency of numerical methods by facilitating the construction of optimal discretizations via adaptive refinement. Furthermore, such estimators offer a means of objectively measuring the quality of a computed numerical solution. A wealth of research exists for the design and theoretical support of effective error estimators in the context of finite-element methods. This includes techniques treating both linear and nonlinear PDEs across a number of applications \cite{John4, Oden1, Verfurth3, Bank1, Babuska2}. As discussed above, the considered optimality conditions represent nonlinear variational systems. Therefore, in deriving the error estimators and proving reliability in this paper, the general framework for nonlinear PDE estimators constructed by Verf\"{u}rth \cite{Verfurth1, Verfurth2} is employed. In order to demonstrate the efficiency and capability of the proposed estimators, they are applied as part of an adaptive mesh refinement (AMR) strategy to a variety of elastic liquid crystal problems. The adapted grids significantly reduce computational work while yielding solutions that are highly competitive with those of uniform mesh in terms of constraint conformance and computed free energies.

This paper is organized as follows. The Frank-Oseen free-energy model and variational systems for the first-order optimality conditions associated with the two constraint enforcement approaches are introduced in Section \ref{Model}. Section \ref{preliminarytheory} discusses additional notation and prerequisite theoretical results to be applied in the reliability proofs to follow. In Section \ref{errorestimator}, error estimators are constructed for both the penalty and Lagrangian formulations of the variational systems and reliability of the penalty method estimator is proven. An AMR strategy applying the derived error estimators is discussed in Section \ref{numerics}, and a number of numerical experiments are performed demonstrating the accuracy and efficiency of the mesh refinement strategy. Finally, Section \ref{conclusions} provides some concluding remarks and a discussion of future work.
 \section{Energy Model and Optimality Conditions} \label{Model}

 While a number of liquid crystal models exist \cite{Onsager1, Davis1, Stewart1}, we consider the Frank-Oseen free-energy model where the equilibrium free energy for a domain $\bar{\Omega}$, with coordinates $\bar{\vec{x}} \in \bar{\Omega}$, is characterized by deformations of the nondimensional unit-length director field, $\director$.  Liquid crystal samples tend towards configurations exhibiting minimal free energy. Let $\bar{K}_i$, $i = 1, 2, 3$ be the Frank constants \cite{Frank1} with $\bar{K}_i \geq 0$ \cite{Ericksen2}. Here, we consider the case that each $\bar{K}_i \neq 0$ and define the dimensionless tensor 
 \begin{align*}
 \vec{Z} = \kappa \director \otimes \director + (\vec{I} - \director \otimes \director) = \vec{I} - (1- \kappa) \director \otimes \director,
 \end{align*}
 where $\kappa = \bar{K}_2/\bar{K}_3$. Note that if $\kappa =1$, $\vec{Z}$ is reduced to the identity. The Frank constants are often anisotropic (i.e., $\bar{K}_1 \neq \bar{K}_2 \neq \bar{K}_3$), vary with liquid crystal type, and play important roles in liquid crystal phenomena \cite{Atherton2, Lee1}. 

We denote the classical $\Ltwo{\Omega}$ inner product and norm as $\Ltwoinner{\cdot}{\cdot}{\Omega}$ and $\Ltwonorm{\cdot}{\Omega}$, respectively, for both scalar and vector quantities. Further, let $(\cdot, \cdot)$ and $\vert \cdot \vert$ denote the Euclidean inner product and norm. Throughout this paper, we assume the presence of Dirichlet boundary conditions and, therefore, utilize the null Lagrangian simplification discussed in \cite{Emerson2, Stewart1}. Thus, the Frank-Oseen free energy for cholesteric liquid crystals is written
 \begin{align} \label{FrankOseenFree}
&\int_{\bar{\Omega}} \bar{w}_F \diff{\bar{V}} = \frac{1}{2} \bar{K}_1 \Ltwonorm{\nabla_{\bar{\vec{x}}} \cdot \director}{\Omega}^2 + \bar{K}_2 \Ltwoinner{\bar{t}_0}{\director \cdot \nabla_{\bar{\vec{x}}} \times \director}{\Omega}  \nonumber \\
 & \qquad + \frac{1}{2} \bar{K}_3\Ltwoinnerndim{\vec{Z} \nabla_{\bar{\vec{x}}} \times \director}{\nabla_{\bar{\vec{x}}} \times \director}{\Omega}{3} + \frac{1}{2} \bar{K}_2 \Ltwoinner{\bar{t}_0}{\bar{t}_0}{\Omega},
 \end{align}
 where $\bar{t}_0$ is the wave parameter characterizing the chiral properties of the cholesteric, which may be positive or negative depending on the handedness of the cholesteric \cite{Collings2}, and $\nabla_{\bar{\vec{x}}}$ represents the standard differential operator for $\bar{\Omega}$. The cholesteric free energy in \eqref{FrankOseenFree} represents a generalization of the standard nematic free energy, discussed in \cite{Stewart1, Emerson1}, and collapses to the nematic representation when $\bar{t}_0 = 0$. Therefore, in deriving the error estimators, the general free-energy model is examined and an estimator for the nematic case is recovered by setting $\bar{t}_0 = 0$.

As noted above, the director field is subject to a local unit-length constraint such that $\director \cdot \director = 1$ at each point throughout the domain. In this paper, we consider enforcing this unit-length constraint, as part of an overall energy-minimization framework, with either a penalty approach or a Lagrange multiplier. In order to properly formulate both methods, we first introduce the following nondimensionalization. Let $\mu$ be a fixed length scale and $K$ denote a characteristic Frank constant. We apply the spatial change of variables $\bar{\vec{x}} = \mu \vec{x}$ to \eqref{FrankOseenFree} and divide the resulting expression by $\mu K$. Finally, define nondimensional Frank constants $K_i = \frac{\bar{K}_i}{K}$, $i=1,2,3$ and $t_0 = \mu \bar{t}_0$. Note that the change of variables also scales derivatives. Thus, to compute free-energy minimizing configurations, we define the nondimensionalized free-energy functional, after rescaling by a factor of $2$, as
 \begin{align}
 \mathcal{G}(\director) &= K_1 \Ltwonorm{\diverg \director}{\Omega}^2 + K_3\Ltwoinnerndim{\vec{Z} \curl \director}{\curl \director}{\Omega}{3} \nonumber \\
 &\qquad + 2 K_2 \Ltwoinner{t_0}{\director \cdot \curl \director}{\Omega} \label{EnergyFunctional}
 \end{align}
 for a dimensionless domain $\Omega$ and differential operator $\nabla$. Note that the $K_2 \Ltwoinner{t_0}{t_0}{\Omega}$ term from \eqref{FrankOseenFree} has been dropped as it does not depend on $\director$ and, as such, may be disregarded in the minimization process.
 
Throughout this paper, it is assumed that $\director \in \Honeb{\Omega}^3 = \{\vec{v} \in \Hone{\Omega}^3 : \vec{v} = \vec{g} \text{ on } \partial \Omega \}$, where $\Hone{\Omega}$ represents the classical Sobolev space with norm $\Honenorm{\cdot}{\Omega}$. Here, we assume that $\vec{g}$ satisfies appropriate compatibility conditions. For example, if $\Omega$ has a Lipschitz continuous boundary, it is assumed that $\vec{g} \in H^{\frac{1}{2}}(\partial \Omega)^3$. Note that if $\vec{g} = \vec{0}$, the space $\Honeb{\Omega}^3 = \Honenot{\Omega}^3$ such that the theory of \cite{Emerson1, Emerson3} is applicable.

The penalty method is constructed by augmenting the functional in \eqref{EnergyFunctional} with a weighted, positive term such that
\begin{align}
\mathcal{H}(\director) &= \mathcal{G}(\director) + \zeta \Ltwoinner{\director \cdot \director -1}{\director \cdot \director -1}{\Omega} \label{Dirichletpenaltyfunctional},
\end{align}
where $\mathcal{H}(\director)$ has been nondimensionalized in the same fashion as the free-energy functional and $\zeta > 0$ represents a constant weight penalizing deviations of the solution from the unit-length constraint. The dimensionless parameter $\zeta$ is defined to be $\zeta = \frac{\mu^2 \bar{\zeta}}{K}$. In the limit of large $\zeta$ values, unconstrained minimization of \eqref{Dirichletpenaltyfunctional} is equivalent to the constrained minimization of \eqref{EnergyFunctional}. To minimize the functional $\mathcal{H}(\director)$, first-order optimality conditions are derived as
\begin{align} \label{penaltyFOOC}
&\mathcal{P}(\director) = K_1\Ltwoinner{\diverg \director}{\diverg \vec{v}}{\Omega} + K_3\Ltwoinnerndim{\vec{Z} \curl \director}{\curl \vec{v}}{\Omega}{3} \nonumber  \\
& \qquad+ (K_2-K_3)\Ltwoinner{\director \cdot \curl \director}{\vec{v} \cdot \curl \director}{\Omega}  \nonumber \\
& \qquad+ K_2 t_0 \big(\Ltwoinner{\vec{v}}{\curl \director}{\Omega} + \Ltwoinner{\director}{\curl \vec{v}}{\Omega} \big) \nonumber \\
& \qquad + 2 \zeta \Ltwoinner{\vec{v} \cdot \director}{\director \cdot \director -1}{\Omega} = 0 \qquad \forall \vec{v} \in \Honenot{\Omega}^3
\end{align}
after canceling coefficients of $2$ for convenience.

An alternative approach to enforcing the unit-length constraint is the use of a nondimensionalized Lagrange multiplier where the Lagrangian is defined as
\begin{align*}
\mathcal{L}(\director, \lambda) &= \mathcal{G}(\director) + \int_{\Omega} \lambda(\vec{x})((\director \cdot \director)-1) \diff{V}.
\end{align*}
Computing the associated first-order optimality conditions yields
\begin{align}
& \mathcal{F}(\director, \lambda) = K_1\Ltwoinner{\diverg \director}{\diverg \vec{v}}{\Omega} + K_3\Ltwoinnerndim{\vec{Z}(\director) \curl \director}{\curl \vec{v}}{\Omega}{3} \nonumber \\
& \quad + (K_2-K_3)\Ltwoinner{\director \cdot \curl \director}{\vec{v} \cdot \curl \director}{\Omega} \nonumber \\
& \quad + K_2 t_0 \big(\Ltwoinner{\vec{v}}{\curl \director}{\Omega} + \Ltwoinner{\director}{\curl \vec{v}}{\Omega} \big) + \int_{\Omega} \lambda \ltwoinner{\director}{\vec{v}} \diff{V} \nonumber \\
& \quad + \int_{\Omega} q (\ltwoinner{\director}{\director} -1) \diff{V} = 0, \label{FOOC}
\end{align}
for all $(\vec{v}, q) \in \Honenot{\Omega}^3 \times \Ltwo{\Omega}$. In the sections to follow, an a posteriori error estimator is derived for both the penalty and Lagrange multiplier methods. Furthermore, theory supporting the reliability of the error estimator for the penalty method is proven. While extending this theory to the estimator for the Lagrangian formulations is the subject of future work, numerical results show that both estimators offer significant performance increases.

\section{Additional Notation and Preliminary Theory} \label{preliminarytheory}

In this section, we consider discretizing the variational systems of Section \ref{Model} with finite elements to approximate equilibrium solutions. In preparation for deriving a posteriori error estimators and the associated supporting theory, some additional notation is defined and requisite preliminary theoretical results are discussed. For the remainder of the paper, it is assumed that $\Omega$ is an open, connected domain of $\mathbb{R}^n$, $n \geq 2$, with polyhedral boundary $\Gamma$. For any open subset $\omega \subset \Omega$ with Lipschitz boundary, the corresponding norms are denoted with an index as $\Vert \cdot \Vert_{1, \omega}$ and $\Vert \cdot \Vert_{0, \omega}$. Furthermore, it is assumed that $\Omega$ is subject to a triangulation with a quasi-uniform family of meshes, $\{\triangulation \}$, for $0 < h \leq 1$, satisfying the conditions
\begin{align}
&\max \{\diam T : T \in \triangulation\} \leq h \, \diam \Omega, \nonumber \\
&\min \{\diam B_T : T \in \triangulation\} \geq \rho h \, \diam \Omega \label{quasiuniform},
\end{align}
where $\rho > 0$ is a constant and $B_T$ is the largest ball contained in a given $T$ such that $T$ is star-shaped with respect to $B_T$. We also require that any two elements of $\triangulation$ are either disjoint or share a complete smooth sub-manifold of their boundaries, satisfying the admissibility property for a triangulation. For any $T \in \triangulation$, let $h_T = \diam T$, denote the set of edges for $T$ as $\mathcal{E}(T)$, and for any $E \in \mathcal{E}(T)$, $h_E = \diam E$. The set $\mathcal{N}(T)$ represents the vertices of $T$, and $\mathcal{N}(E)$ is the set of vertices for $E$. The complete set of edges and vertices, respectively, for a triangulation $\triangulation$ is 
\begin{align*}
\mathcal{E}_h = \bigcup_{T \in \triangulation} \mathcal{E}(T), & & \mathcal{N}_h = \bigcup_{T \in \triangulation} \mathcal{N}(T),
\end{align*}
with $\mathcal{E}_{h, \Omega}$ signifying the set of edges excluding those on the boundary, $\Gamma$. It is also assumed that the mesh family is fine enough such that $h_T, h_E \leq 1$.

Note that the quasi-uniformity condition in \eqref{quasiuniform} ensures that for all $T \in \triangulation$ and $E \in \mathcal{E}(T)$, there exist constants such that $C_{l} \leq h_T/h_E \leq C_{u}$ independent of $h$, $T$, and $E$ \cite{Verfurth2}. Furthermore, it implies that the smallest angle of any $T$ is bounded from below by a constant independent of $h$. Finally, let
\begin{align*}
\omega_T &= \bigcup_{\mathcal{E}(T) \cap \mathcal{E}(T') \neq \emptyset} T', & \omega_E &= \bigcup_{E \in \mathcal{E}(T')} T',\\
\tilde{\omega}_T &= \bigcup_{\mathcal{N}(T) \cap \mathcal{N}(T') \neq \emptyset} T', & \tilde{\omega}_E &= \bigcup_{\mathcal{N}(E) \cap \mathcal{N}(T') \neq \emptyset} T'.
\end{align*}
Each of the quantities above represent subdomains of $\Omega$. Let $\hat{T} = \{\hat{x} \in \mathbb{R}^n : \sum_{i = 1}^n \hat{x}_i \leq 1, \hat{x}_j \geq 0, 1 \leq j \leq n \}$ denote a fixed reference element and $\hat{E} = \hat{T} \cap \{ \hat{x} \in \mathbb{R}^n : \hat{x}_n = 0 \}$ a fixed reference edge for the triangulation. The triangulation is assumed to be affine equivalent in the sense that, for any $T \in \triangulation$, there exists an invertible affine mapping from the reference components to $T$.
For any $E \in \mathcal{E}_{h, \Omega}$ and piecewise continuous function $\phi$, the jump of $\phi$ across $E$ in the direction $\eta_E$ is denoted $[\phi]_E$. Finally, for a given $k \in \mathbb{N}$, define the space
\begin{align*}
S_h^{k, 0} &= \{\phi: \Omega \rightarrow \Rone : \phi \vert_T \in \Pi_k,  \forall T \in \triangulation \} \cap C(\bar{\Omega})
\end{align*}
where $\Pi_k$ is the set of polynomials of degree at most $k$, let $\phi \vert_T$ denote the restriction of $\phi$ to the element $T$, and set $C(\bar{\Omega})$ as the collection of continuous functions on the closure of $\Omega$.

With the notation established above, a number of important supporting results are gathered here and referenced substantially in the reliability proofs of Section \ref{errorestimator}. The first is an approximation error bound for the Cl\'{e}ment interpolation operator \cite{Clement1, Verfurth1}. Denoting the operator $I_h: L^1(\Omega) \rightarrow S_h^{1, 0}$, the following holds for $\triangulation$.

\begin{lemma} \label{clementlemma}
For any $T \in \triangulation$ and $E \in \mathcal{E}_h$
\begin{align*}
\Vert \phi - I_h \phi \Vert_{0, T} &\leq C_1 h_T \Vert \phi \Vert_{1, \tilde{\omega}_T} & \forall \phi \in \Hone{\tilde{\omega}_T}, \\
\Vert \phi - I_h \phi \Vert_{0, E} &\leq C_2 h_E^{1/2} \Vert \phi \Vert_{1, \tilde{\omega}_E} & \forall \phi \in \Hone{\tilde{\omega}_E},
\end{align*}
where $C_1$ and $C_2$ depend only on the quasi-uniformity condition in \eqref{quasiuniform}. 
\end{lemma}

Following the notation in \cite{Verfurth1, Verfurth2}, let $\psi_{\hat{T}}, \psi_{\hat{E}} \\ \in C^{\infty}(\hat{T}, \Rone)$ be cut-off functions such that
\begin{align*}
&0 \leq \psi_{\hat{T}} \leq 1, \quad \max_{\hat{x} \in \hat{T}} \psi_{\hat{T}}(\hat{x}) = 1, \quad \psi_{\hat{T}} = 0 \text{ on } \partial \hat{T}, \\
&0 \leq \psi_{\hat{E}} \leq 1, \quad \max_{\hat{x} \in \hat{E}} \psi_{\hat{E}}(\hat{x}) = 1, \quad \psi_{\hat{E}} = 0 \text{ on } \partial \hat{T} \backslash \hat{E}.
\end{align*}
Define a continuation operator $\hat{P}: L^{\infty}(\hat{E}) \rightarrow L^{\infty}(\hat{T})$ as
\begin{align*}
\hat{P} \hat{u}(\hat{x}_1, \ldots, \hat{x}_n) &:= \hat{u}(\hat{x}_1, \ldots, \hat{x}_{n-1})
\end{align*}
for all $\hat{x} \in \hat{T}, \hat{u} \in L^{\infty}(\hat{E})$, and fix $V_{\hat{T}} \subset L^{\infty}(\hat{T})$ and $V_{\hat{E}} \subset L^{\infty}(\hat{E})$ as two arbitrary finite-dimensional subspaces. Using the affine mappings from reference components to triangulation components, the corresponding functions, $\psi_T$ and $\psi_E$, operator $P: L^{\infty}(E) \rightarrow L^{\infty}(T)$, and spaces $V_T$ and $V_E$ are extended to arbitrary $T \in \triangulation$ and $E \in \mathcal{E}_h$ while preserving the properties discussed above. With these definitions the following lemma and corollary hold, c.f. \cite{Verfurth1, Verfurth2, Brenner1}.

\begin{lemma} \label{cutoffinequalities}
There are constants $C_1, \ldots, C_7$ depending only on the finite-dimensional spaces $V_{\hat{T}}$ and $V_{\hat{E}}$, the functions $\psi_{\hat{T}}$ and $\psi_{\hat{E}}$, and  the quasi-uniform bound of \eqref{quasiuniform} such that for all $T \in \triangulation$, $E \in \mathcal{E}(T)$, $u \in V_T$, and $\sigma \in V_E$
\begin{align}
C_1 \Vert u \Vert_{0, T} &\leq \sup_{v \in V_T} \frac{\int_T u \psi_T v \diff{V}}{\Vert v \Vert_{0, T}} \leq \Vert u \Vert_{0, T}, \label{cutoff1}\\
C_2 \Vert \sigma \Vert_{0, E} &\leq \sup_{\tau \in V_E} \frac{\int_E \sigma \psi_E \tau \diff{S}}{\Vert \tau \Vert_{0, E}} \leq \Vert \sigma \Vert_{0, E}, \label{cutoff2}\\
C_3 h_T^{-1} \Vert \psi_T u \Vert_{0, T} &\leq \Vert \nabla (\psi_T u) \Vert_{0, T} \nonumber \\
& \hspace{3.5em} \leq C_4 h_T^{-1} \Vert \psi_T u \Vert_{0, T}, \label{cutoff3} \\
C_5 h_T^{-1} \Vert \psi_E P\sigma \Vert_{0, T} &\leq \Vert \nabla (\psi_E P \sigma) \Vert_{0, T} \nonumber \\
& \hspace{3.5em} \leq C_6 h_T^{-1} \Vert \psi_E P \sigma \Vert_{0, T}, \label{cutoff4} \\
\Vert \psi_E P \sigma \Vert_{0, T} &\leq C_7 h_T^{1/2} \Vert \sigma \Vert_{0, E}. \label{cutoff5}
\end{align}
\end{lemma}
Note that with the quasi-uniformity of the triangulation, after proper adjustment of $C_i$ in any of the above inequalities, the mesh constant $h_T$ may be exchanged for $h_E$ while maintaining the inequality.

\begin{corollary} \label{cutoffexpansion}
Under the assumptions of Lemma \ref{cutoffinequalities}, there exists a $\bar{C}_4 > 0$ and $\bar{C}_6 > 0$ such that
\begin{align}
\Vert \psi_T u  \Vert_{1,T} &\leq \bar{C}_4 h_T^{-1} \Vert \psi_T u \Vert_{0, T}, \label{expansion1} \\
\Vert \psi_E P \sigma \Vert_{1, T} &\leq \bar{C}_6 h_T^{-1} \Vert \psi_E P \sigma \Vert_{0, T}. \label{expansion2}
\end{align}
\end{corollary}
\begin{proof}
These inequalities follow directly from Inequalities \eqref{cutoff3} and \eqref{cutoff4}, respectively, and the fact that $h_T \leq 1$.
\end{proof}

Finally, we state two central propositions of Verf\"{u}rth \cite{Verfurth1, Verfurth2}. Consider Banach spaces $X$ and $Y$ with respective norms $\Vert \cdot \Vert_X$ and $\Vert \cdot \Vert_Y$. Let $\mathcal{L}(X, Y)$ correspond to the space of continuous linear maps from $X$ to $Y$ with the natural operator norm $\Vert \cdot \Vert_{\mathcal{L}(X, Y)}$. Further, let $\text{Isom}(X, Y)$ be the set of linear homeomorphisms from $X$ onto $Y$. Define $Y^* = \mathcal{L}(Y, \Rone)$ as the dual space of $Y$ and denote the associated duality pairing as $\langle \cdot, \cdot \rangle$. Let $F \in C^1(X, Y^*)$ be a continuously differentiable function for which a solution $u \in X$ is sought such that
\begin{align}
F(u) = 0. \label{nonlinearProblem}
\end{align}
The derivative of $F$ is written as $DF$. For any $u \in X$ and any real number $R > 0$, the ball centered at $u$ with radius $R$ is defined as $B(u, R) = \{ v \in X : \Vert u - v \Vert_X < R \}$.
\begin{proposition} \label{nonlinearErrorEstimation}
Let $u_0 \in X$ be a regular solution for Equation \eqref{nonlinearProblem} in the sense that $DF(u_0) \in \text{Isom}(X, Y^*)$. Assume that $DF$ is Lipschitz continuous at $u_0$, where there exists an $R_0 > 0$ such that
\begin{align*}
\gamma = \sup_{u \in B(u_0, R_0)} \frac{\Vert DF(u) - DF(u_0) \Vert_{\mathcal{L}(X, Y^*)}}{\Vert u - u_0 \Vert_X} < \infty.
\end{align*}
Set
\begin{align*}
R &= \min \{ R_0, \gamma^{-1} \Vert DF(u_0)^{-1} \Vert^{-1}_{\mathcal{L}(Y^*, X)}, \nonumber \\ 
& \hspace{10em} 2 \gamma^{-1} \Vert DF(u_0) \Vert_{\mathcal{L}(X, Y^*)} \}.
\end{align*}
Then the error estimate
\begin{align*}
& \frac{1}{2} \Vert DF(u_0) \Vert_{\mathcal{L}(X, Y^*)}^{-1} \Vert F(u) \Vert_{Y^*} \leq \Vert u - u_0 \Vert_X \nonumber \\
& \hspace{7em} \leq 2 \Vert DF(u_0)^{-1} \Vert_{\mathcal{L}(Y^*, X)} \Vert F(u) \Vert_{Y^*}
\end{align*}
holds for all $u \in B(u_0, R)$.
\end{proposition}

Let $X_h \subset X$ and $Y_h \subset Y$ be finite-dimensional subspaces and $F_h \in C(X_h, Y_h^*)$ be an approximation of $F$. Consider the discretized problem of finding $u_h \in X_h$ such that
\begin{align}
F_h(u_h) = 0. \label{finitedimNonlinearProblem}
\end{align}
\vspace{-1.5em}
\begin{proposition} \label{auxiliarySpaceInequality}
Let $u_h \in X_h$ be an approximate solution for Equation \eqref{finitedimNonlinearProblem} in the sense that $\Vert F_h(u_h) \Vert_{Y_h^*}$ is ``small''. Assume that there is a restriction operator $R_h \in \mathcal{L}(Y, Y_h)$, a finite-dimensional space $\tilde{Y}_h \subset Y$, and an approximation $\tilde{F}_h : X_h \rightarrow Y^*$ of $F$ at $u_h$ such that 
\begin{align*}
\Vert (\text{Id}_Y - R_h)^*\tilde{F}_h(u_h) \Vert_{Y^*} \leq C_0 \Vert \tilde{F}_h (u_h) \Vert_{\tilde{Y}_h^*},
\end{align*}
where $\text{Id}_Y$ is the identity operator on $Y$ and $C_0 > 0$ is independent of $h$. Then the following estimate holds.
\begin{align*}
\Vert F(u_h) \Vert_{Y^*} &\leq C_0 \Vert \tilde{F}_h (u_h) \Vert_{\tilde{Y}^*_h} \nonumber \\
& \hspace{1em} + \Vert (\text{Id}_Y - R_h)^*[F(u_h) - \tilde{F}_h(u_h)] \Vert_{Y^*} \\
& \hspace{1em} + \Vert R_h \Vert_{\mathcal{L}(Y, Y_h)} \Vert F(u_h) - F_h(u_h) \Vert_{Y_h^*} \nonumber \\
& \hspace{1em} + \Vert R_h \Vert_{\mathcal{L}(Y, Y_h)} \Vert F_h(u_h) \Vert_{Y_h^*}.
\end{align*}
\end{proposition}

\section{A Posteriori Error Estimators} \label{errorestimator}

In this section, an a posteriori error estimator is derived for the nonlinear variational problem in \eqref{penaltyFOOC}, representing the first-order optimality conditions of the penalty method discussed in Section \ref{Model}, and theory supporting its reliability as an estimator is shown. In addition, we propose an error estimator for the first-order optimality conditions in \eqref{FOOC} associated with the Lagrange multiplier approach. 

Considering the first-order optimality conditions for the penalty method in \eqref{penaltyFOOC}, set $Y^* = \left (\Honenot{\Omega}^3 \right )^*$ and $X = \Honeb{\Omega}^3$. Therefore, $\mathcal{P}(\director) \in C^1(X, Y^*)$. In order to construct an approximate solution to \eqref{penaltyFOOC}, we consider a general discretization of the form
\begin{align*}
[S_h^{1, 0}]^3 \subset V_h \subset [S_h^{s, 0}]^3,
\end{align*}
for $s \geq 1$ and define the finite-dimensional space $Y_h = \{\vec{v}_h \in V_h : \vec{v}_h = 0 \text{ on } \Gamma \}$. For the theory presented here, we assume that the boundary conditions for $\director$ are exactly representable in the finite-element space $V_h$ on the coarsest grid of $\{\triangulation\}$. Note that this restriction on the boundary conditions admits projection of the boundary function $\vec{g}$ onto the coarsest mesh of $\{\triangulation\}$. At each level of refinement, the boundary conditions are projected onto the refined mesh and, therefore, the analysis to follow applies to the additionally refined levels. Thus, the analysis herein estimates the error arising from discrete approximation of solutions to \eqref{penaltyFOOC} on the interior of $\Omega$ but not from approximation of the boundary conditions. Hence, set $X_h = V_h \cap X$. For $\vec{v} \in Y$ and $\langle \mathcal{P}(\director), \vec{v} \rangle$, define
\begin{align*}
\langle \mathcal{P}_h(\director_h), \vec{v}_h \rangle = \langle \mathcal{P}(\director_h), \vec{v}_h \rangle,
\end{align*}
where $\director_h \in X_h$ and $\vec{v}_h \in Y_h$.

Let $\director_h$ be a solution of
\begin{align} \label{penaltyDiscreteFOOC}
\mathcal{P}_h(\director_h) = 0 & & \forall \vec{v}_h \in Y_h.
\end{align}
For each $T \in \triangulation$, denoting the outward facing normal of $E \in \mathcal{E}(T)$ as $\eta_E$, integrating by parts elementwise, and using the fact that $\vec{v} = 0$ on $\Gamma$ yields
\begin{widetext}
\begin{align*}
\langle \mathcal{P}(\director_h), \vec{v} \rangle &= K_1 \sum_{T \in \triangulation} \int_T -\nabla (\diverg \director_h) \cdot \vec{v} \diff{V} + K_1 \sum_{E \in \mathcal{E}_{h, \Omega}} \int_E [(\diverg \director_h) \eta_E]_E \cdot \vec{v} \diff{S} \nonumber \\
& \quad + K_3 \sum_{T \in \triangulation} \int_T \curl (\vec{Z}(\director_h) \curl \director_h) \cdot \vec{v} \diff{V} + K_3 \sum_{E \in \mathcal{E}_{h, \Omega}} \int_E [(\vec{Z}(\director_h) \curl \director_h) \times \eta_E]_E \cdot \vec{v} \diff{S} \nonumber \\
& \quad + (K_2 - K_3) \sum_{T \in \triangulation} \int_T ((\director_h \cdot \curl \director_h) \curl \director_h) \cdot \vec{v} \diff{V} + 2 K_2 t_0 \sum_{T \in \triangulation} \int_T (\curl \director_h) \cdot \vec{v} \diff{V} \nonumber \\
& \quad + 2 \zeta \sum_{T \in \triangulation} \int_T ((\director_h \cdot \director_h - 1) \director_h) \cdot \vec{v}  \diff{V}.
\end{align*}
\end{widetext}
Define a restriction operator $R_h: Y \rightarrow Y_h$ as $R_h[\vec{u}] = [I_h u_1, I_h u_2, I_h u_3]$ where $I_h$ is the Cl\'{e}ment operator of Lemma \ref{clementlemma}. As there is no forcing function or Neumann boundary conditions and the Dirichlet boundary is exactly captured by the finite-element space, set
\begin{align*}
\langle \tilde{\mathcal{P}}_h(\director_h), \vec{v} \rangle = \langle \mathcal{P}(\director_h), \vec{v} \rangle.
\end{align*}
Note that this immediately implies that for the quantities from Proposition \ref{auxiliarySpaceInequality}:
\begin{align}
\Vert (\text{Id}_Y - R_h)^*[\mathcal{P}(\director_h) - \tilde{\mathcal{P}}_h(\director_h)] \Vert_{Y^*} &= 0, \label{penaltyzero1} \\
\Vert \mathcal{P}(\director_h) - \mathcal{P}_h(\director_h) \Vert_{Y_h^*} &= 0. \label{penaltyzero2}
\end{align}
For any $T \in \triangulation$, define
\begin{widetext}
\begin{align*}
\Theta_T &= \Bigg \{ h_T^2 \big \Vert -K_1 \nabla (\diverg \director_h) + K_3 \curl (\vec{Z}(\director_h) \curl \director_h) + (K_2 - K_3) (\director_h \cdot \curl \director_h) \curl \director_h + 2K_2 t_0 (\curl \director_h) \nonumber \\
& \qquad + 2 \zeta (\director_h \cdot \director_h - 1) \director_h \big \Vert_{0, T}^2 + \sum_{E \in \mathcal{E}(T) \cap \mathcal{E}_{h, \Omega}} h_E \big \Vert [K_1 (\diverg \director_h) \eta_E  + K_3 (\vec{Z}(\director_h) \curl \director_h) \times \eta_E]_E \big \Vert_{0, E}^2 \Bigg \}^{1/2}.
\end{align*}
\end{widetext}
As is shown below, $\Theta_T$ constitutes a reliable and constructible local error estimator for each element of the discretization. 

For the definitions of $\tilde{\mathcal{P}}(\director_h)$ and $\Theta_T$ above, the following lemma holds. \\ \newpage
\begin{lemma} \label{penaltyupperBoundRestriction} 
There exists a constant $C > 0$ independent of $h$ such that
\begin{align*}
\Vert (\text{Id}_Y - R_h)^* \tilde{\mathcal{P}}_h(\director_h) \Vert_{Y^*} \leq C \left( \sum_{T \in \triangulation} \Theta_T^2 \right)^{1/2}.
\end{align*}
\end{lemma}
\begin{proof}
First note that
\begin{widetext}
\vspace{-1.5em}
\begin{align}
& \Vert (\text{Id}_Y - R_h)^* \tilde{\mathcal{P}}_h(\director_h) \Vert_{Y^*} = \sup_{\substack{\vec{v} \in Y \\ \Vert \vec{v} \Vert_Y = 1}} \sum_{T \in \triangulation} \sum_{i = 1}^3 \int_T \Big( -K_1 (\nabla(\diverg \director_h))_i + K_3(\curl (\vec{Z}(\director_h) \curl \director_h))_i \nonumber \\
& \hspace{5em}  + (K_2 - K_3) ((\director_h \cdot \curl \director_h) \curl \director_h)_i  + 2 K_2 t_0  (\curl \director_h)_i + 2 \zeta ((\director_h \cdot \director_h - 1) \director_h)_i \Big) \big( v_i - I_h v_i \big) \diff{V} \nonumber \\
& \hspace{5em}  + \sum_{E \in \mathcal{E}_{h, \Omega}} \sum_{i=1}^3 \int_E [K_1 ((\diverg \director_h) \eta_E)_i + K_3 ((\vec{Z}(\director_h) \curl \director_h) \times \eta_E)_i]_E \cdot (v_i - I_h v_i) \diff{S}. \label{PenaltyRestrictionSupDefinition}
\end{align}
Applying the Cauchy-Schwarz inequality and Lemma \ref{clementlemma} to each component in \eqref{PenaltyRestrictionSupDefinition},
\begin{align}
&\Vert (\text{Id}_Y - R_h)^* \tilde{\mathcal{P}}_h(\director_h) \Vert_{Y^*} \leq  \sup_{\substack{\vec{v} \in Y \\ \Vert \vec{v} \Vert_Y = 1}} \sum_{T \in \triangulation} \sum_{i = 1}^3 \bigg \{ \big \Vert -K_1 (\nabla(\diverg \director_h))_i + K_3(\curl (\vec{Z}(\director_h) \curl \director_h))_i \nonumber \\
& \hspace{4em}  + (K_2 - K_3) ((\director_h \cdot \curl \director_h) \curl \director_h)_i + 2 K_2 t_0  (\curl \director_h)_i  + 2 \zeta ((\director_h \cdot \director_h - 1) \director_h)_i \big \Vert_{0, T} \cdot C_1 h_T \Vert v_i \Vert_{1, \tilde{\omega}_T} \bigg \} \nonumber \\
& \hspace{4em} + \sum_{E \in \mathcal{E}_{h, \Omega}} \sum_{i=1}^3 \big \Vert [K_1 ((\diverg \director_h) \eta_E)_i + K_3 ((\vec{Z}(\director_h) \curl \director_h) \times \eta_E)_i]_E \big \Vert_{0, E} \cdot  C_2 h_E^{1/2} \Vert v_i \Vert_{1, \tilde{\omega}_E} \nonumber \\
& \hspace{10.5em} \leq  \sup_{\substack{\vec{v} \in Y \\ \Vert \vec{v} \Vert_Y = 1}} \max(C_1, C_2) \Bigg ( \sum_{T \in \triangulation} \sum_{i = 1}^3 \bigg \{ h_T^2 \big \Vert -K_1 (\nabla(\diverg \director_h))_i + K_3(\curl (\vec{Z}(\director_h) \curl \director_h))_i \nonumber \\
& \qquad \qquad + (K_2 - K_3) ((\director_h \cdot \curl \director_h) \curl \director_h)_i + 2 K_2 t_0  (\curl \director_h)_i + 2 \zeta ((\director_h \cdot \director_h - 1) \director_h)_i \big \Vert_{0, T}^2 \bigg \} \nonumber \\
& \hspace{4em} + \sum_{E \in \mathcal{E}_{h, \Omega}} \sum_{i=1}^3 h_E \big \Vert [K_1 ((\diverg \director_h) \eta_E)_i + K_3 ((\vec{Z}(\director_h) \curl \director_h) \times \eta_E)_i]_E \big \Vert_{0, E}^2 \Bigg )^{1/2} \nonumber \\
& \hspace{4em} \cdot \Bigg( \sum_{T \in \triangulation} \sum_{i=1}^3 \Vert v_i \Vert^2_{1, \tilde{\omega}_T} + \sum_{E \in \mathcal{E}_{h, \Omega}} \sum_{i=1}^3 \Vert v_i \Vert_{1, \tilde{\omega}_E}^2 \Bigg )^{1/2} \label{CSInequalityForSums}
\end{align}
where the last inequality of \eqref{CSInequalityForSums} is given by the Cauchy-Schwarz inequality for sums. Finally, there exists $C_* > 0$, independent of $h$ and taking into account repeated elements in the sums, such that
\begin{align*}
\left ( \sum_{T \in \triangulation} \Vert w \Vert_{1, \tilde{w}_T}^2 + \sum_{E \in \mathcal{E}_{h, \Omega}} \Vert w \Vert_{1, \tilde{w}_E}^2 \right)^{1/2} \leq C_* \Vert w \Vert_1.
\end{align*}
Applying the above inequality to \eqref{CSInequalityForSums} yields
\begin{align*}
&\Vert (\text{Id}_Y - R_h)^* \tilde{\mathcal{P}}_h(\director_h) \Vert_{Y^*} \\
& \leq \sup_{\substack{\vec{v} \in Y \\ \Vert \vec{v} \Vert_Y = 1}} C_* \max(C_1, C_2) \Vert \vec{v} \Vert_Y \Bigg ( \sum_{T \in \triangulation} \sum_{i = 1}^3 \bigg \{ h_T^2 \big \Vert -K_1 (\nabla(\diverg \director_h))_i  + K_3(\curl (\vec{Z}(\director_h) \curl \director_h))_i \nonumber \\
& \qquad \qquad + (K_2 - K_3) ((\director_h \cdot \curl \director_h) \curl \director_h)_i + 2 K_2 t_0  (\curl \director_h)_i + 2 \zeta ((\director_h \cdot \director_h - 1) \director_h)_i \big \Vert_{0, T}^2 \bigg \} \nonumber \\
& \qquad \qquad + \sum_{E \in \mathcal{E}_{h, \Omega}} \sum_{i=1}^3 h_E \big \Vert [K_1 ((\diverg \director_h) \eta_E)_i + K_3 ((\vec{Z}(\director_h) \curl \director_h) \times \eta_E)_i]_E \big \Vert_{0, E}^2 \Bigg )^{1/2}.
\end{align*}
\end{widetext}
Noting that the jump components in the bound are summed over $E \in \mathcal{E}_{h, \Omega}$, this implies that
\begin{align*}
\Vert (\text{Id}_Y - R_h)^* \tilde{\mathcal{P}}_h(\director_h) \Vert_{Y^*} \leq C \left ( \sum_{T \in \triangulation} \Theta_T^2 \right)^{1/2}.
\end{align*}
\end{proof}

Next, define the subspace $\tilde{Y}_h \subset Y$ as
\begin{align*}
\tilde{Y}_h = \text{span } \{\psi_T \vec{v}, &\text{ } \psi_E P \sigma : \vec{v} \in [\Pi_{k \vert T}]^3, \nonumber \\
&\sigma \in [\Pi_{k \vert E}]^3, \text{ } T \in \triangulation, E \in \mathcal{E}_{h, \Omega} \},
\end{align*}
such that $k \geq 3s$, $\psi_T$ and $\psi_E$ are cutoff functions, and $P$ is the continuation operator as defined in Section \ref{preliminarytheory}. For this subspace, the lemma below holds.
\begin{lemma} \label{penaltyupperBoundTildeY}
There exists a constant $C > 0$ independent of $h$ such that
\begin{align*}
\Vert \tilde{\mathcal{P}}_h(\director_h) \Vert_{\tilde{Y}_h^*} \leq C \left (\sum_{T \in \triangulation} \Theta_T^2 \right)^{1/2}.
\end{align*}
\end{lemma}
\begin{proof}
Observe that
\begin{widetext}
\vspace{-1.5em}
\begin{align*}
\Vert \tilde{\mathcal{P}}_h(\director_h) \Vert_{\tilde{Y}_h^*}  &= \sup_{\substack{\vec{v_h} \in \tilde{Y}_h \\ \Vert \vec{v}_h \Vert_Y = 1}} \sum_{T \in \triangulation} \Bigg \{ \int_T \bigg( -K_1 \nabla(\diverg \director_h) + K_3 \curl (\vec{Z}(\director_h) \curl \director_h) \nonumber \\
& \qquad + (K_2 - K_3) (\director_h \cdot \curl \director_h) \curl \director_h + 2 K_2 t_0 \curl \director_h + 2 \zeta (\director_h \cdot \director_h - 1) \director_h \bigg) \cdot \vec{v}_h \diff{V} \Bigg \} \nonumber \\
& \qquad + \sum_{E \in \mathcal{E}_{h, \Omega}} \int_E [K_1 (\diverg \director_h) \eta_E + K_3 (\vec{Z}(\director_h) \curl \director_h) \times \eta_E]_E \cdot \vec{v}_h \diff{S} \nonumber \\
& \leq \sup_{\substack{\vec{v}_h \in \tilde{Y}_h \\ \Vert \vec{v}_h \Vert_Y = 1}} \sum_{T \in \triangulation} \Bigg \{ \big \Vert -K_1 \nabla(\diverg \director_h) + K_3 \curl (\vec{Z}(\director_h) \curl \director_h) \nonumber \\
& \qquad + (K_2 - K_3) (\director_h \cdot \curl \director_h) \curl \director_h + 2 K_2 t_0 \curl \director_h + 2 \zeta (\director_h \cdot \director_h - 1) \director_h \big \Vert_{0, T} \Vert \vec{v}_h \Vert_{0, T} \Bigg \} \nonumber \\
& \qquad + \sum_{E \in \mathcal{E}_{h, \Omega}} \big \Vert [K_1 (\diverg \director_h) \eta_E + K_3 (\vec{Z}(\director_h) \curl \director_h) \times \eta_E]_E \big \Vert_{0, E} \Vert \vec{v}_h \Vert_{0, E}
\end{align*}
applying the Cauchy-Schwarz inequality. Using the definition of $\tilde{Y}_h$, the quasi-uniformity of $\triangulation$, and standard finite-element scaling arguments
\begin{align}
\Vert \tilde{\mathcal{P}}_h(\director_h) \Vert_{\tilde{Y}_h^*} &\leq \sup_{\substack{\vec{v}_h \in \tilde{Y}_h \\ \Vert \vec{v}_h \Vert_Y = 1}} \sum_{T \in \triangulation} \Bigg \{ C_1 h_T \big \Vert -K_1 \nabla(\diverg \director_h) + K_3 \curl (\vec{Z}(\director_h) \curl \director_h) \nonumber \\
&  \qquad + (K_2 - K_3) (\director_h \cdot \curl \director_h) \curl \director_h + 2 K_2 t_0 \curl \director_h + 2 \zeta (\director_h \cdot \director_h - 1) \director_h \big \Vert_{0, T} \Vert \vec{v}_h \Vert_{1, T}  \Bigg \} \nonumber \\
& \qquad + \sum_{E \in \mathcal{E}_{h, \Omega}} C_2 h_E^{1/2} \big \Vert [K_1 (\diverg \director_h) \eta_E + K_3 (\vec{Z}(\director_h) \curl \director_h) \times \eta_E]_E \big \Vert_{0, E} \Vert \vec{v}_h \Vert_{1, \omega_E} \nonumber \\
& \leq \sup_{\substack{\vec{v}_h \in \tilde{Y}_h \\ \Vert \vec{v}_h \Vert_Y = 1}} \max(C_1, C_2) \Bigg( \sum_{T \in \triangulation} h_T^2 \big \Vert -K_1 \nabla(\diverg \director_h) + K_3 \curl (\vec{Z}(\director_h) \curl \director_h) \nonumber \\
& \qquad + (K_2 - K_3) (\director_h \cdot \curl \director_h) \curl \director_h + 2 K_2 t_0 \curl \director_h + 2 \zeta (\director_h \cdot \director_h - 1) \director_h \big \Vert_{0, T}^2 \nonumber \\
& \qquad + \sum_{E \in \mathcal{E}_{h, \Omega}} h_E \big \Vert [K_1 (\diverg \director_h) \eta_E + K_3 (\vec{Z}(\director_h) \curl \director_h) \times \eta_E]_E \big \Vert_{0, E}^2 \Bigg)^{1/2} \nonumber \\
& \qquad \cdot \Bigg( \sum_{T \in \triangulation} \Vert \vec{v}_h \Vert^2_{1, T} + \sum_{E \in \mathcal{E}_{h, \Omega}} \Vert \vec{v}_h \Vert^2_{1, \omega_E} \Bigg)^{1/2}, \label{partialFtildeBound}
\end{align}
where the second inequality comes after applying the Cauchy-Schwarz inequality for sums. As in the previous proof, there exists a $C_* > 0$, independent of $h$ and taking into account repeated elements in the sum, such that
\begin{align}
\bigg( \sum_{T \in \triangulation} \Vert \vec{v}_h \Vert^2_{1, T} + \sum_{E \in \mathcal{E}_{h, \Omega}} \Vert \vec{v}_h \Vert^2_{1, \omega_E} \bigg)^{1/2} \leq C_* \Vert \vec{v}_h \Vert_Y. \label{sumTriagleBound}
\end{align}
\end{widetext}
Combining \eqref{partialFtildeBound} and \eqref{sumTriagleBound} implies that
\begin{align*}
\Vert \tilde{\mathcal{P}}_h(\director_h) \Vert_{\tilde{Y}_h^*} \leq C \left ( \sum_{T \in \triangulation} \Theta_T^2 \right)^{1/2},
\end{align*}
where $C$ depends only on the dimension of the problem, choice of reference elements, and the smallest angle of the triangulation.
\end{proof}

The final inequality targeted is showing that there exists a $C > 0$, independent of $h$, such that
\begin{align*}
\Vert (\text{Id}_Y - R_h)^* \tilde{\mathcal{P}}_h (\director_h) \Vert_{Y^*} \leq C \Vert \tilde{\mathcal{P}}_h(\director_h) \Vert_{\tilde{Y}_h^*}.
\end{align*}
Given the result in Lemma \ref{penaltyupperBoundRestriction}, it suffices to prove the following lemma.
\begin{lemma} \label{penaltylowerBoundTildeF}
There exists a $C > 0$, independent of $h$, such that
\begin{align*}
C \left( \sum_{T \in \triangulation} \Theta_T^2 \right)^{1/2} \leq \Vert \tilde{\mathcal{P}}_h(\director_h) \Vert_{\tilde{Y}_h^*}.
\end{align*}
\end{lemma}
\begin{proof}
Consider an arbitrary element $T \in \triangulation$ and edge $E \in \mathcal{E}(T) \cap \mathcal{E}_{h, \Omega}$ and define the space $\tilde{Y}_{h \vert \omega}$, for $\omega \in \{ T, \omega_E, \omega_T \}$, as the set of all functions $\phi \in \tilde{Y}_h$ with $\text{supp}(\phi) \subset \omega$. Note that in this proof the numbered constants correspond to those of Lemmas \ref{cutoffinequalities} and \ref{cutoffexpansion}. First,
\begin{widetext}
\begin{align}
&C_1 \bar{C}_4^{-1} h_T \big \Vert -K_1 \nabla (\diverg \director_h) + K_3 \curl (\vec{Z}(\director_h) \curl \director_h) + (K_2 - K_3) (\director_h \cdot \curl \director_h) \curl \director_h \nonumber \\
& \hspace{5em} + 2 K_2 t_0 \curl \director_h + 2 \zeta (\director_h \cdot \director_h - 1) \director_h \big \Vert_{0, T} \nonumber \\
& \leq \sup_{\vec{w} \in [\Pi_{k \vert T}]^3 \backslash \{\vec{0}\}} \bar{C}_4^{-1} h_T \Vert \psi_T \vec{w} \Vert^{-1}_{0, T} \int_T \Big( -K_1 \nabla(\diverg \director_h) + K_3 \curl (\vec{Z}(\director_h) \curl \director_h) \nonumber \\
& \hspace{5em} + (K_2 - K_3) (\director_h \cdot \curl \director_h) \curl \director_h + 2 K_2 t_0 \curl \director_h + 2 \zeta (\director_h \cdot \director_h - 1) \director_h \Big) \cdot  \psi_T \vec{w} \diff{V} \label{penaltysecondcomponent1} \\
& \leq \sup_{\vec{w} \in [\Pi_{k \vert T}]^3 \backslash \{\vec{0}\}} \Vert \psi_T \vec{w} \Vert^{-1}_{1, T} \int_T \Big( -K_1 \nabla(\diverg \director_h) + K_3 \curl (\vec{Z}(\director_h) \curl \director_h) \nonumber \\
& \hspace{5em} + (K_2 - K_3) (\director_h \cdot \curl \director_h) \curl \director_h + 2 K_2 t_0 \curl \director_h  + 2 \zeta (\director_h \cdot \director_h - 1) \director_h \Big) \cdot \psi_T \vec{w} \diff{V} \label{penaltysecondcomponent2}\\
& \leq \sup_{\vec{w} \in [\Pi_{k \vert T}]^3 \backslash \{\vec{0}\}} \Vert \psi_T \vec{w} \Vert^{-1}_{1, T} \langle \tilde{\mathcal{P}}_h(\director_h), \psi_T \vec{w} \rangle \nonumber \\
& \leq \sup_{\substack{\vec{v}_h \in \tilde{Y}_{h \vert T} \\ \Vert \vec{v}_h \Vert_Y = 1}} \langle \tilde{\mathcal{P}}_h(\director_h), \vec{v} \rangle. \label{penaltysecondcomponent4}
\end{align}
\end{widetext}
Inequality \eqref{penaltysecondcomponent1} is given by \eqref{cutoff1} from Lemma \ref{cutoffinequalities}, and \eqref{penaltysecondcomponent2} is a consequence of \eqref{expansion1} from Lemma \ref{cutoffexpansion}. The inequality in \eqref{penaltysecondcomponent4} comes from expanding the space over which the supremum is taken and noting that $\psi_T \vec{w}$ vanishes at the boundary of $T$.

Recall that, with proper consideration, the constant $h_T$ in Lemma \ref{cutoffinequalities} is exchangeable for $h_E$. Now we bound
\begin{widetext}
\vspace{-1em}
\begin{align}
&C_2 \bar{C}_6^{-1} C_7^{-1} h_E^{1/2} \big \Vert [K_1 (\diverg \director_h) \eta_E + K_3 (\vec{Z}(\director_h) \curl \director_h) \times \eta_E]_E \big \Vert_{0, E} \nonumber \\
& \leq \sup_{\sigma \in [\Pi_{k \vert E}]^3 \backslash \{\vec{0} \}} \frac{\bar{C}_6^{-1} h_E}{C_7 h_E^{1/2} \Vert P \sigma \Vert_{0, E}} \int_E [K_1(\diverg \director_h) \eta_E + K_3(\vec{Z}(\director_h) \curl \director_h) \times \eta_E]_E \cdot \psi_E P \sigma \diff{S} \label{penaltythirdcomponent1} \\
& = \sup_{\sigma \in [\Pi_{k \vert E}]^3 \backslash \{\vec{0} \}} \frac{\bar{C}_6^{-1} h_E}{C_7 h_E^{1/2} \Vert \sigma \Vert_{0, E}} \bigg \{ \langle \tilde{\mathcal{P}}_h(\director_h), \psi_E P \sigma \rangle - \int_{\omega_E} \Big (-K_1 \nabla (\diverg \director_h) + K_3 \curl (\vec{Z}(\director_h) \curl \director_h) \nonumber \\
& \hspace{5em} + (K_2 - K_3) (\director_h \cdot \curl \director_h) \curl \director_h + 2 K_2 t_0 \curl \director_h + 2 \zeta (\director_h \cdot \director_h - 1) \director_h \Big)  \cdot \psi_E P \sigma \diff{V} \bigg \}, \label{penaltythirdcomponent2}
\end{align}
\end{widetext}
where \eqref{penaltythirdcomponent1} is given by \eqref{cutoff2} of Lemma \ref{cutoffinequalities} and \eqref{penaltythirdcomponent2} comes from the fact that $\psi_E P \sigma$ is supported on $\omega_E$ and that the norm in the denominator is over an edge and the continuation operator, $P$, does not modify the values of $\sigma$ there. Next, we apply \eqref{cutoff5} of Lemma \ref{cutoffinequalities} for each of the elements in $\omega_E$ and distribute the fraction quantity. Using \eqref{expansion2} from Lemma \ref{cutoffexpansion} for each of the elements in $\omega_E$ for the first summand and the Cauchy-Schwarz inequality along with cancellation of the resulting $\Vert \psi_E P \sigma \Vert_{0, \omega_E}$ terms in the second summand yields,
\begin{widetext}
\vspace{-1em}
\begin{align*}
&C_2 \bar{C}_6^{-1} C_7^{-1} h_E^{1/2} \big \Vert [K_1 (\diverg \director_h) \eta_E + K_3 (\vec{Z}(\director_h) \curl \director_h) \times \eta_E]_E \big \Vert_{0, E} \\
& \hspace{2em} \leq \sup_{\sigma \in [\Pi_{k \vert E}]^3 \backslash \{\vec{0} \}} \frac{\bar{C}_6^{-1} h_E}{\Vert \psi_E P \sigma \Vert_{0, \omega_E}} \bigg \{ \langle \tilde{\mathcal{P}}_h(\director_h), \psi_E P \sigma \rangle - \int_{\omega_E} \Big (-K_1 \nabla (\diverg \director_h) + K_3 \curl (\vec{Z}(\director_h) \curl \director_h) \nonumber \\
& \hspace{6em} + (K_2 - K_3) (\director_h \cdot \curl \director_h) \curl \director_h + 2 K_2 t_0 \curl \director_h + 2 \zeta (\director_h \cdot \director_h - 1)\director_h \Big) \cdot \psi_E P \sigma \diff{V} \bigg \}  \\
& \hspace{2em} \leq \sup_{\sigma \in [\Pi_{k \vert E}]^3 \backslash \{\vec{0} \}} \Vert \psi_E P \sigma \Vert_{1, \omega_E}^{-1} \langle \tilde{\mathcal{P}}_h(\director_h), \psi_E P \sigma \rangle + \bar{C}_6^{-1} h_E \big \Vert -K_1 \nabla (\diverg \director_h) + K_3 \curl (\vec{Z}(\director_h) \curl \director_h) \nonumber \\
& \hspace{6em} + (K_2 - K_3) (\director_h \cdot \curl \director_h) \curl \director_h + 2 K_2 t_0 \curl \director_h + 2 \zeta (\director_h \cdot \director_h - 1) \director_h \big \Vert_{0, \omega_E}.
\end{align*}
This then implies that 
\begin{align}
&C_2 \bar{C}_6^{-1} C_7^{-1} h_E^{1/2} \big \Vert [K_1 (\diverg \director_h) \eta_E + K_3 (\vec{Z}(\director_h) \curl \director_h) \times \eta_E]_E \big \Vert_{0, E} \nonumber \\
& \hspace{13em} \leq \sup_{\substack{\vec{v}_h \in \tilde{Y}_{h \vert {\omega_E}} \\ \Vert \vec{v}_h \Vert_Y = 1}} \langle \tilde{\mathcal{P}}_h(\director_h), \vec{v}_h \rangle + C_d \sup_{\substack{\vec{v}_h \in \tilde{Y}_{h \vert {\omega_E}} \\ \Vert \vec{v}_h \Vert_Y = 1}} \langle \tilde{\mathcal{P}}_h(\director_h), \vec{v}_h \rangle. \label{penaltythirdcomponent5}
\end{align}
\end{widetext}
The first part of the inequality is given by expanding the space over which the supremum is taken. The second component of \eqref{penaltythirdcomponent5} uses the inequality in \eqref{penaltysecondcomponent4}, where $C_d$ relates the constants $\bar{C}_6^{-1} h_E$ to $C_1 \bar{C}_4^{-1} h_T$. 

Observe that the bounds in \eqref{penaltysecondcomponent4} and \eqref{penaltythirdcomponent5} only get larger when considering a supremum over $\tilde{Y}_{h \vert {\omega_T}}$. Hence, gathering the bounds in \eqref{penaltysecondcomponent4} and \eqref{penaltythirdcomponent5} yields
\begin{align*}
\bar{C} \Theta_T \leq \sup_{\substack{\vec{v}_h \in \tilde{Y}_{h \vert {\omega_T}} \\ \Vert \vec{v}_h \Vert_Y = 1}} \langle \tilde{\mathcal{P}}_h(\director_h), \vec{v}_h \rangle
\end{align*}
for $\bar{C}$ independent of $h$. Finally, recall that if $a_i \geq 0$,
\begin{align}
\left(\sum_i a_i \right)^{1/2} \leq \sum_i a_i^{1/2}. \label{sqrtTriangle}
\end{align} 
Summing over $T \in \triangulation$ and applying \eqref{sqrtTriangle} yields
\begin{align*}
C \left( \sum_{T \in \triangulation} \Theta_T^2 \right)^{1/2} \leq \Vert \tilde{\mathcal{P}}_h(\director_h) \Vert_{\tilde{Y}_h^*}.
\end{align*}
\end{proof}

With the preceding lemmas established, we now state and prove the main result of this section. 
\begin{theorem} \label{penaltyaposteriorierrorestimator}
Say that $\director_*$ is a solution to Equation \eqref{penaltyFOOC} satisfying the assumptions of Proposition \ref{nonlinearErrorEstimation}. Furthermore, let $\director_h$ be a solution to the discrete problem, as in Equation \eqref{penaltyDiscreteFOOC}, such that $\Vert \mathcal{P}_h(\director_h) \Vert_{Y_h^*} = 0$ and $\director_h \in B(\director_*, R)$. Then there exists a $C>0$, independent of $h$, such that
\begin{align*}
\Vert \director_* - \director_h \Vert_1 \leq C \left( \sum_{T \in \triangulation} \Theta_T^2 \right)^{1/2}.
\end{align*}
\end{theorem}
\begin{proof}
Combining the results of Lemmas \ref{penaltyupperBoundRestriction} and \ref{penaltylowerBoundTildeF} implies that
\begin{align*}
\Vert (\text{Id}_Y - R_h)^* \tilde{\mathcal{P}}_h(\director_h) \Vert_{Y^*} &\leq \tilde{C}_1 \left( \sum_{T \in \triangulation} \Theta_T^2 \right)^{1/2} \nonumber \\
& \leq \tilde{C}_2 \Vert \tilde{\mathcal{P}}_h(\director_h) \Vert_{\tilde{Y}_h^*},
\end{align*}
for $\tilde{C}_2 > 0$, independent of $h$. Thus, the conditions of Proposition \ref{auxiliarySpaceInequality} are fulfilled. This implies, noting the quantities in Equations \eqref{penaltyzero1} and \eqref{penaltyzero2}, that
\begin{align*}
\Vert \mathcal{P}(\director_h) \Vert_{Y^*} \leq C_0 \Vert \tilde{\mathcal{P}}_h(\director_h) \Vert_{\tilde{Y}_h^*} \leq C_1 \left( \sum_{T \in \triangulation} \Theta_T^2 \right)^{1/2},
\end{align*}
where the second inequality is a consequence of Lemma \ref{penaltyupperBoundTildeY}. Finally, using the upper bound from Proposition \ref{nonlinearErrorEstimation} implies
\begin{align*}
&\Vert \director_* - \director_h \Vert_1 \leq 2 \Vert D\mathcal{P}(\director_*)^{-1} \Vert_{\mathcal{L}(Y^*, X)} \Vert \mathcal{P}(\director_h) \Vert_{Y^*}  \nonumber \\
& \hspace{4em} \leq 2 C_1 \Vert D\mathcal{P}(\director_*)^{-1} \Vert_{\mathcal{L}(Y^*, X)} \left( \sum_{T \in \triangulation} \Theta_T^2 \right)^{1/2}.
\end{align*}
Setting $C = 2 C_1 \Vert D\mathcal{P}(\director_*)^{-1} \Vert_{\mathcal{L}(Y^*, X)}$ yields the desired inequality.
\end{proof}

As discussed in \cite[Remark 2.5]{Verfurth1}, the results of Lemma \ref{cutoffinequalities} are equally applicable to quadrilateral or simplicial elements. Thus, the theory above holds for mesh families of quadrilateral finite elements satisfying equivalent conditions, as used in the numerical experiments below.

Considering the Lagrange multiplier variational form in Equation \eqref{FOOC} and following an analogous process to the penalty case, a related element-wise estimator is derived. For a $T \in \triangulation$,
\begin{widetext}
\vspace{-1.5em}
\begin{align}
\Theta_T &= \Bigg \{ h_T^2 \big \Vert -K_1 \nabla (\diverg \director_h) + K_3 \curl (\vec{Z}(\director_h) \curl \director_h) + (K_2 - K_3) (\director_h \cdot \curl \director_h) \curl \director_h + 2K_2 t_0 (\curl \director_h) \nonumber \\
& + \lambda_h \director_h \big \Vert_{0, T}^2  + \Vert \director_h \cdot \director_h - 1 \Vert_{0, T}^2 + \sum_{E \in \mathcal{E}(T) \cap \mathcal{E}_{h, \Omega}} h_E \big \Vert [K_1 (\diverg \director_h) \eta_E + K_3 (\vec{Z}(\director_h) \curl \director_h) \times \eta_E]_E \big \Vert_{0, E}^2 \Bigg \}^{1/2}. \label{lagrangeEstimator}
\end{align}
\end{widetext}
While theory establishing invertibility of the discretized derivative of the variational form in \eqref{FOOC} has been demonstrated in \cite{Emerson1}, solution pairs $(\director_*, \lambda_*)$ satisfying the constraint are not localized in the sense that $\lambda_*$ may be freely perturbed while the pair remains a solution. Therefore, though a number of the lemmas above can be extended to the estimator proposed in \eqref{lagrangeEstimator}, special theoretical treatment is required to properly apply the propositions of Section \ref{preliminarytheory} and is the subject of future work. Nevertheless, the experiments in Section \ref{numerics} report and discuss the performance of the estimator numerically.


\section{Numerical Results} \label{numerics}

In this section, the results of a number of simulations applying the a posteriori error estimators derived above are discussed. In general, the algorithm to compute equilibrium solutions to the nonlinear variational problems presented in Section \ref{Model} has three stages; see Algorithm \ref{algo}. The outermost phase is nested iteration (NI) \cite{Starke1}, which begins on a specified coarsest grid. Newton iterations are performed on each grid, updating the solution approximation at each step, as in \cite{Emerson1}. The stopping criterion for the Newton iterations at each level is based on a specified tolerance for the current approximation's conformance to the first-order optimality conditions in the standard Euclidean $l_2$ norm. In the numerical experiments to follow, this tolerance is fixed at $10^{-4}$. The resulting approximation is then interpolated to a finer grid. The composition of the finer mesh is determined via an AMR strategy based on the value of the appropriate error estimator, $\Theta_T$, for each element $T$ of the coarser mesh. Given the coarse approximate solution $\vec{u}_H$, $\Theta_T$ is computed for each $T$ of mesh $H$. In the numerical simulations to follow, the top $40\%$ of elements are then refined based on their estimator values. For each simulation, the characteristic length scale discussed above is taken to be one micron, such that $\mu = 10^{-6}$ m. Furthermore, the characteristic Frank constant is taken to be $K = 6.2 \times 10^{-12}$ N, the dimensional value of $K_1$ for $5$CB, a common liquid crystal.

\vspace{.3cm}
\begin{algorithm}[H] \label{algo}
\SetAlgoLined
~\\
0. Initialize $\vec{u}_0$ on coarse grid.
~\\
\While{Refinement limit not reached}
{
	\While{First-order optimality conformance threshold not satisfied}
	{
		1. Set up discrete linearized system for Newton iterations on grid $H$. ~\\
		2. Solve for $\delta \vec{u}_H$. ~\\
		3. Compute $\vec{u}_{k+1}$ as in $\vec{u}_k + \alpha \delta \vec{u}_H$. ~\\
	}
	4. Compute $\Theta_T$ on each element for approximate solution $\vec{u}_H$. ~\\
	5. Adaptively refine the grid. ~\\
	6. Interpolate $\vec{u}_H \to \vec{u}_h$.
}
\caption{Newton's method with NI and AMR}
\end{algorithm}
\vspace{.3cm}

For each Newton iteration an incomplete Newton correction is performed such that for a given iterate $\vec{u}_k$, the next Newton iterate is given by $\vec{u}_{k+1} = \vec{u}_k + \alpha \delta \vec{u}_h$, where $\alpha \leq 1$. As discussed in \cite{Emerson1}, this is to encourage strict adherence to the constraint manifold associated with the unit-length requirement imposed on the director field. At each level of NI, the damping parameter $\alpha$ is increased, to a maximum of $1.0$, as the finer features of the solution become increasingly resolved on finer mesh. The grid management, discretizations, and adaptive refinement computations are implemented with the widely used deal.II finite-element library \cite{BangerthHartmannKanschat2007}. In the simulations below, $Q_2$ elements are used to approximate components associated with $\director$, and $Q_1$ elements are applied for computations involving $\lambda$, where appropriate, on each grid. 

The significant nonlinearity present in the Frank-Oseen energy model has limited known analytical solutions in the presence of Dirichlet boundary conditions, especially for two and three dimensional domains. Therefore, while analytical error rates are not available for the problems considered in this section, a number of other metrics are reported to demonstrate the performance of the AMR strategy.

In order to quantify the efficiency gains with adaptive refinement, the simulations applying AMR are compared with those using uniform meshes at each level. To compare the differences in computational work required across the NI sequences, an approximate work unit (WU) is calculated. Assuming the presence of solvers that scale linearly with the number of non-zeros in the matrix, a WU is defined as the sum of the non-zeros in the discretized Hessian for each Newtons step divided by the number of non-zeros in a discretized Hessian on the finest grid. The total roughly approximates the work required by the full NI hierarchy in terms of assembling and solving a single linearization step on the finest level. To compare the work on the uniformly refined mesh, the WUs there are computed using the total number of non-zeros of the finest mesh resulting from the associated simulation applying AMR. Therefore, the WUs reported in the experiments to follow are approximate measures of the required work in terms of assembling and solving a linearization step on the finest adaptively refined grid when optimally scaling solvers are applied. While the linear systems here are solved with simple LU decomposition, the reported WUs provide a best-case scaling baseline for comparing the work required with the refinement strategies.

The meshes considered in these numerical simulations utilize quadrilateral elements. Adaptive refinement, then, gives rise to the existence of hanging nodes. These nodes are dealt with in the standard way by constraining their values with the neighboring regular nodes to maintain continuity along the boundary. Additionally, we constrain the number of hanging nodes on a single edge to a maximum of one. If refinement of a given element would violate this requirement, the adjacent node is also refined to enforce the limit. Therefore, in a given refinement step, though a fixed percentage of cells are flagged for refinement, additional elements may be refined. Finally, the theory developed in the preceding sections assumes that the studied mesh satisfies the admissibility property. This requirement is fully satisfied for the coarsest mesh but, with the introduction of hanging nodes, no longer holds after the first AMR stage. While mesh discretizations employing triangular simplices can maintain admissibility with adaptivity, grids composed purely of quadrilateral or cuboid elements cannot. Thus, following the first level of refinement, the error estimator is applied heuristically to guide AMR on the remaining levels.

\subsection{Twisting Nematics in a Square Domain} \label{TwistSection}

\begin{table*}[t]
\centering
\resizebox{\textwidth}{!}{
\begin{tabular}{|c|c|c|c|c|c|c||c|c|c|c|c|c|}
\hlinewd{1.3pt}
 & \multicolumn{3}{|c|}{Penalty (Adapt.)} & \multicolumn{3}{|c||}{Penalty (Uniform)} & \multicolumn{3}{c|}{Lagrangian (Adapt.)} & \multicolumn{3}{|c|}{Lagrangian (Uniform)} \\
\hlinewd{1.3pt}
Level & Energy & Max. Dev. & Min. Dev. & Energy & Max. Dev. & Min. Dev. & Energy & Max. Dev. & Min. Dev. & Energy & Max. Dev. & Min. Dev. \\
\hlinewd{1.3pt}
Grid $1$ & $3.640$ & $-2.714$e-$06$ & $-8.786$e-$05$ & $3.640$ & $-2.714$e-$06$ & $-8.786$e-$05$ & $3.641$ & $3.447$e-$06$ & $-2.689$e-$06$ & $3.641$ & $3.447$e-$06$ & $-2.689$e-$06$ \\
\hline
Grid $2$ & $3.640$ & $-2.954$e-$06$ & $-8.420$e-$05$ & $3.640$ & $-2.954$e-$06$ & $-8.420$e-$05$ & $3.641$ & $2.167$e-$06$ & $-1.962$e-$06$ & $3.641$ & $4.308$e-$07$ & $-5.939$e-$07$ \\
\hline
Grid $3$ & $3.640$ & $-2.110$e-$06$ & $-7.244$e-$05$ & $3.640$ & $-2.110$e-$06$ & $-7.244$e-$05$ & $3.641$ & $3.663$e-$07$ & $-3.266$e-$07$ & $3.641$ & $1.139$e-$07$ & $-1.485$e-$07$ \\
\hline
Grid $4$ & $3.640$ & $-1.108$e-$06$ & $-6.468$e-$05$ & $3.640$ & $-1.108$e-$06$ & $-6.468$e-$05$ & $3.641$ & $2.722$e-$07$ & $-2.482$e-$07$ & $3.641$ & $3.637$e-$08$ & $-4.215$e-$08$ \\
\hline
Grid $5$ & $3.640$ & $-4.521$e-$07$ & $-6.469$e-$05$ & $3.640$ & $-4.521$e-$07$ & $-6.469$e-$05$ & $3.641$ & $4.128$e-$08$ & $-4.030$e-$08$ & $3.641$ & $1.177$e-$08$ & $-1.342$e-$08$ \\
\hline
Grid $6$ & $3.640$ & $-1.573$e-$07$ & $-6.469$e-$05$ & $3.640$ & $-1.573$e-$07$ & $-6.469$e-$05$ & $3.641$ & $2.953$e-$08$ & $-2.824$e-$08$ & $3.641$ & $3.826$e-$09$ & $-4.483$e-$09$ \\
\hlinewd{1.3pt}
Fine DOF & \multicolumn{3}{|c|}{$691, 923$} & \multicolumn{3}{|c||}{$12, 595, 203$} & \multicolumn{3}{c|}{$692, 386$} & \multicolumn{3}{|c|}{$13, 645, 828$} \\
\hline
WUs & \multicolumn{3}{|c|}{$5.224$} & \multicolumn{3}{|c||}{$36.221$} & \multicolumn{3}{c|}{$4.040$} & \multicolumn{3}{|c|}{$30.109$} \\
\hline
Timing & \multicolumn{3}{|c|}{$337.7$s} & \multicolumn{3}{|c||}{$11,904.1$s} & \multicolumn{3}{c|}{$490.3$s} & \multicolumn{3}{|c|}{$15,682.4$s} \\
\hlinewd{1.3pt}
\end{tabular}
}
\caption{\small{Simulation statistics associated with the nematic simulations for the square domain with twisting conditions for both adaptive and uniformly refined mesh hierarchies. The column Fine DOF reflects the number of degrees of freedom on the finest mesh of the NI. The largest director deviations above and below unit length at the quadrature nodes are shown in the Max. Dev. and Min Dev. columns.}}
\label{TwistExperimentStatistics}
\end{table*}

\begin{figure}[h!]
\begin{minipage}{0.99 \linewidth}
\subfloat[]{
\includegraphics[scale=.22, center]{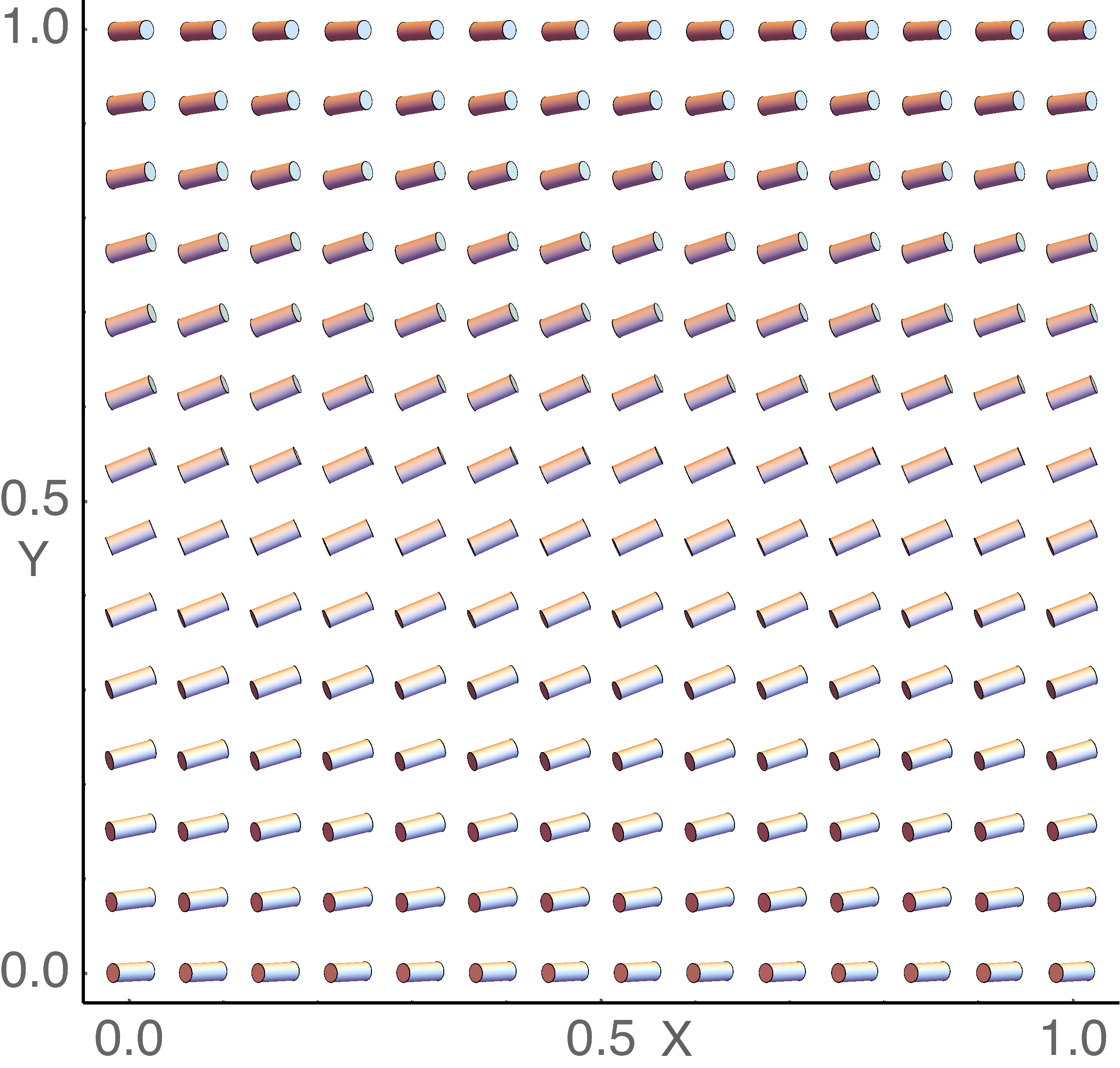}
\label{TwistSquareTwistEnergyDensityPlots:left}
}
\end{minipage}
\begin{minipage}{0.49 \linewidth}
\subfloat[]{ 
\includegraphics[scale=.113, center]{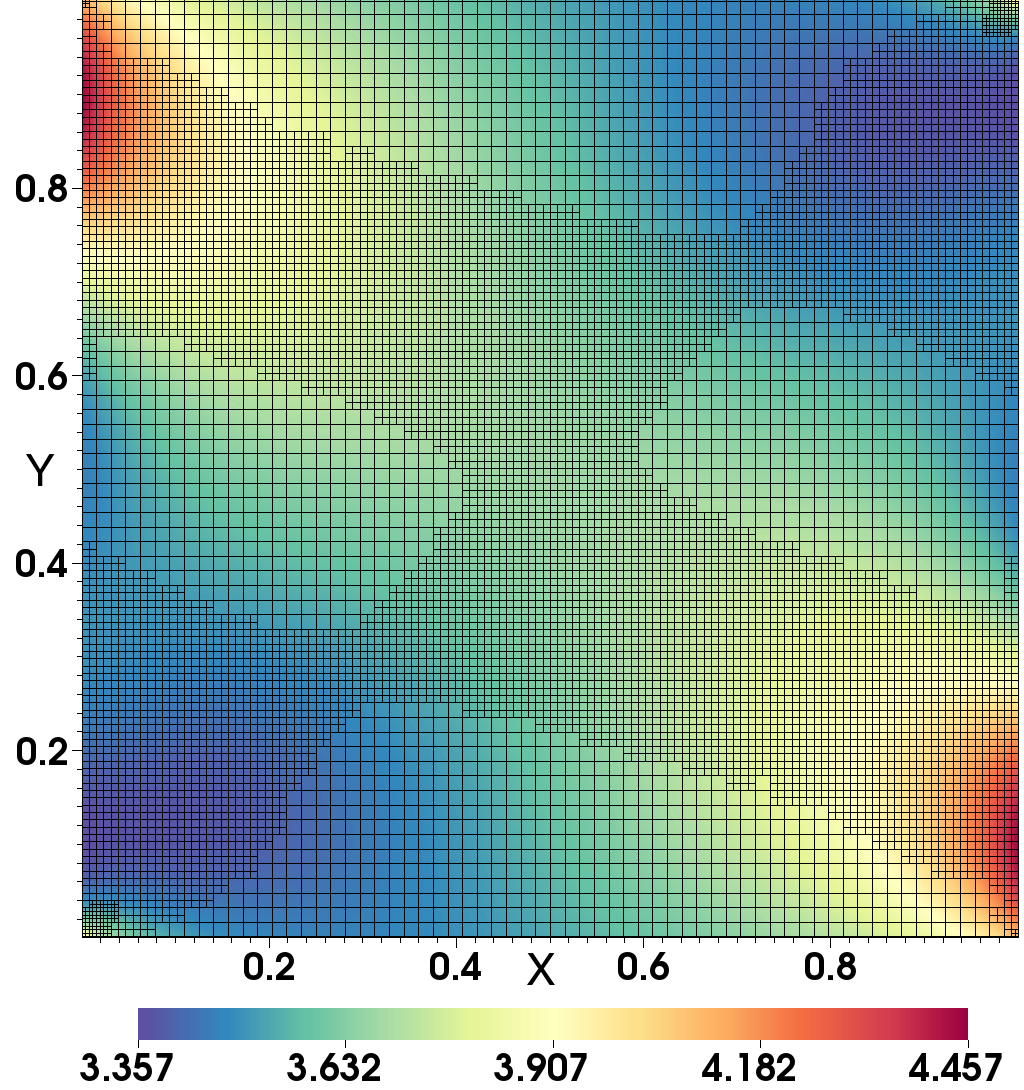}
\label{TwistSquareTwistEnergyDensityPlots:center}
}
\end{minipage}
\begin{minipage}{0.49 \linewidth}
\subfloat[]{
\includegraphics[scale=.113, center]{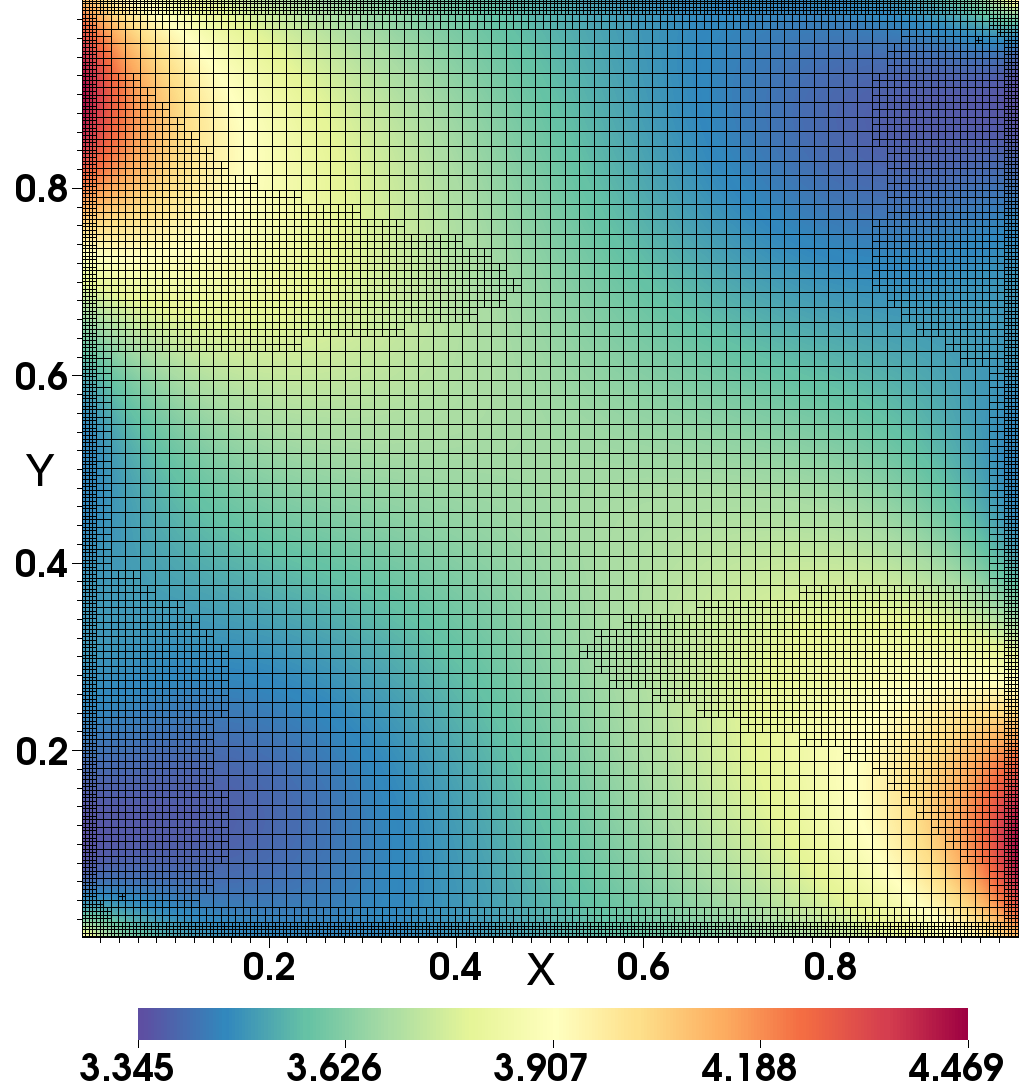}
\label{TwistSquareTwistEnergyDensityPlots:right}
}
\end{minipage}
\caption{\small{\subref{TwistSquareTwistEnergyDensityPlots:left} The final computed solution for the Lagrangian formulation on the adapted mesh (restricted for visualization). Frank-Oseen energy-density function, $w_F$, with overlaid AMR patterns after three refinements for the \subref{TwistSquareTwistEnergyDensityPlots:center} Lagrange multiplier approach and \subref{TwistSquareTwistEnergyDensityPlots:right} penalty method.}}
\label{TwistSquareTwistEnergyDensityPlots}
\end{figure}

The first set of numerical experiments examines a unit-square domain with tilt-twist-type boundary conditions as shown in Figure \ref{TwistSquareTwistEnergyDensityPlots}\subref{TwistSquareTwistEnergyDensityPlots:left}. Along the top and bottom boundaries, the director is rotated counter-clockwise with respect to the positive $x$-axis in the $xz$-plane by a constant angle of $-\frac{\pi}{8}$ at $y=0$ and $\frac{\pi}{8}$ at $y=1$. At the left and right boundaries the director twists to reconcile the director orientations of upper and lower boundaries while also tilting to a maximum angle of $\frac{\pi}{4}$ in the direction of the $y$-axis. The Frank constants for these simulations are set to $K_1 = 1.0$, $K_2 = 3.0$, and $K_3 = 1.2$ and $t_0 = 0$. These parameters are chosen because they have been shown to predispose the nematic towards tilt even in the presence of periodic boundary conditions which do not strongly encourage such a response on the interior of the domain \cite{Emerson3, Leslie2}. The NI mesh hierarchy begins on a uniform $32 \times 32$ mesh and proceeds with five additional levels of either adaptive or uniform mesh refinements. The damping parameter starts at $\alpha = 0.2$ and increases by $0.2$ after each refinement for the Lagrange multiplier approach, and $\alpha = 0.4$ with increases of $0.2$ for the penalty method. Finally, the penalty parameter is set to $\zeta = 10^5$ for experiments applying the penalty method.

Figures \ref{TwistSquareTwistEnergyDensityPlots}\subref{TwistSquareTwistEnergyDensityPlots:center} and \subref{TwistSquareTwistEnergyDensityPlots:right} display the computed free-energy density for the director configuration shown in Figure \ref{TwistSquareTwistEnergyDensityPlots}\subref{TwistSquareTwistEnergyDensityPlots:left}. The director field exhibits significant uniformity in tilt throughout the interior of the domain in agreement with expectations based on the chosen Frank constants. Overlaid on the energy density plots are the AMR patterns resulting from three successive refinement stages based on the error estimator for the Lagrange multiplier approach and the penalty method. While the refinement structures differ with the method applied, regions of emphasis are shared that generally coincide with areas of elevated free energy. The penalty error estimator places more refinement emphasis along the boundary of the domain, with coarser refinement along the interior compared to the Lagrange multiplier estimator. 

Table \ref{TwistExperimentStatistics} details a number of statistics comparing the performance of AMR for each method against a uniform refinement strategy. At the finest level, the uniformly refined mesh includes nearly 20 times more degrees of freedom and requires run times over $30$ times longer to reach the specified tolerance for both approaches. The computed free energies are in good agreement across all meshes. For the Lagrangian method, the unit-length constraint conformance with the AMR technique is competitive with the uniform refinement strategy but slightly looser. On the other hand, AMR for the penalty method maintains nearly identical constraint enforcement compared with the uniform mesh. In either case, indistinguishable solutions are acquired with AMR drastically reducing the computational work necessary to obtain them.

\subsection{Nematics in a Patterned Square Domain}

\begin{table*}[t]
\centering
\resizebox{\textwidth}{!}{
\begin{tabular}{|c|c|c|c|c|c|c||c|c|c|c|c|c|}
\hlinewd{1.3pt}
 & \multicolumn{3}{|c|}{Penalty (Adapt.)} & \multicolumn{3}{|c||}{Penalty (Uniform)}  & \multicolumn{3}{c|}{Lagrangian (Adapt.)} & \multicolumn{3}{|c|}{Lagrangian (Uniform)} \\
\hlinewd{1.3pt}
Level & Energy & Max. Dev. & Min. Dev. & Energy & Max. Dev. & Min. Dev. & Energy & Max. Dev. & Min. Dev. & Energy & Max. Dev. & Min. Dev. \\
\hlinewd{1.3pt}
Grid $1$ & $10.925$ & $2.481$e-$02$ & $-4.144$e-$02$ & $10.925$ & $2.481$e-$02$ & $-4.144$e-$02$ & $10.866$ & $2.920$e-$02$ & $-3.852$e-$02$ & $10.866$ & $2.920$e-$02$ & $-3.852$e-$02$ \\
\hline
Grid $2$ & $10.876$ & $1.207$e-$02$ & $-1.707$e-$02$ & $10.876$ & $1.207$e-$02$ & $-1.707$e-$02$ & $10.875$ & $1.310$e-$02$ & $-1.569$e-$02$ & $10.875$ & $1.310$e-$02$ & $-1.569$e-$02$ \\
\hline
Grid $3$ & $10.866$ & $2.785$e-$03$ & $-3.391$e-$03$ & $10.866$ & $2.785$e-$03$ & $-3.391$e-$03$ & $10.866$ & $3.057$e-$03$ & $-3.213$e-$03$ & $10.866$ & $3.057$e-$03$ & $-3.213$e-$03$ \\
\hline
Grid $4$ & $10.866$ & $8.562$e-$04$ & $-5.825$e-$04$ & $10.866$ & $8.562$e-$04$ & $-5.825$e-$04$ & $10.866$ & $4.064$e-$04$ & $-4.878$e-$04$ & $10.866$ & $4.064$e-$04$ & $-4.878$e-$04$ \\
\hline
Grid $5$ & $10.866$ & $6.805$e-$04$ & $-3.390$e-$04$ & $10.866$ & $6.805$e-$04$ & $-3.390$e-$04$ & $10.866$ & $7.904$e-$05$ & $-9.432$e-$05$ & $10.866$ & $7.904$e-$05$ & $-9.432$e-$05$ \\
\hline
Grid $6$ & $10.867$ & $5.962$e-$04$ & $-2.918$e-$04$ & $10.867$ & $5.962$e-$04$ & $-2.918$e-$04$ & $10.867$ & $7.837$e-$05$ & $-9.416$e-$05$ & $10.867$ & $7.837$e-$05$ & $-9.416$e-$05$ \\
\hlinewd{1.3pt}
Fine DOF & \multicolumn{3}{|c|}{$653,019$} & \multicolumn{3}{|c||}{$12,595,203$} & \multicolumn{3}{c|}{$700, 203$} & \multicolumn{3}{|c|}{$13,645,828$} \\
\hline
WUs & \multicolumn{3}{|c|}{$12.306$} & \multicolumn{3}{|c||}{$69.640$} & \multicolumn{3}{c|}{$6.864$} & \multicolumn{3}{|c|}{$63.672$} \\
\hline
Timing & \multicolumn{3}{|c|}{$583.7$s} & \multicolumn{3}{|c||}{$21,177.6$s} & \multicolumn{3}{c|}{$600.3$s} & \multicolumn{3}{|c|}{$48,383.3$s} \\
\hlinewd{1.3pt}
\end{tabular}
}
\caption{\small{Simulation statistics associated with the nematic simulations for the patterned square domain for both adaptive and uniformly refined mesh hierarchies. The column Fine DOF reflects the number of degrees of freedom on the finest mesh of the NI. The largest director deviations above and below unit length at the quadrature nodes are shown in the Max. Dev. and Min Dev. columns.}}
\label{PatternedExperimentStatistics}
\end{table*}

\begin{figure}[h!]
\begin{minipage}{0.99 \linewidth}
\subfloat[]{
\includegraphics[scale=.205, center]{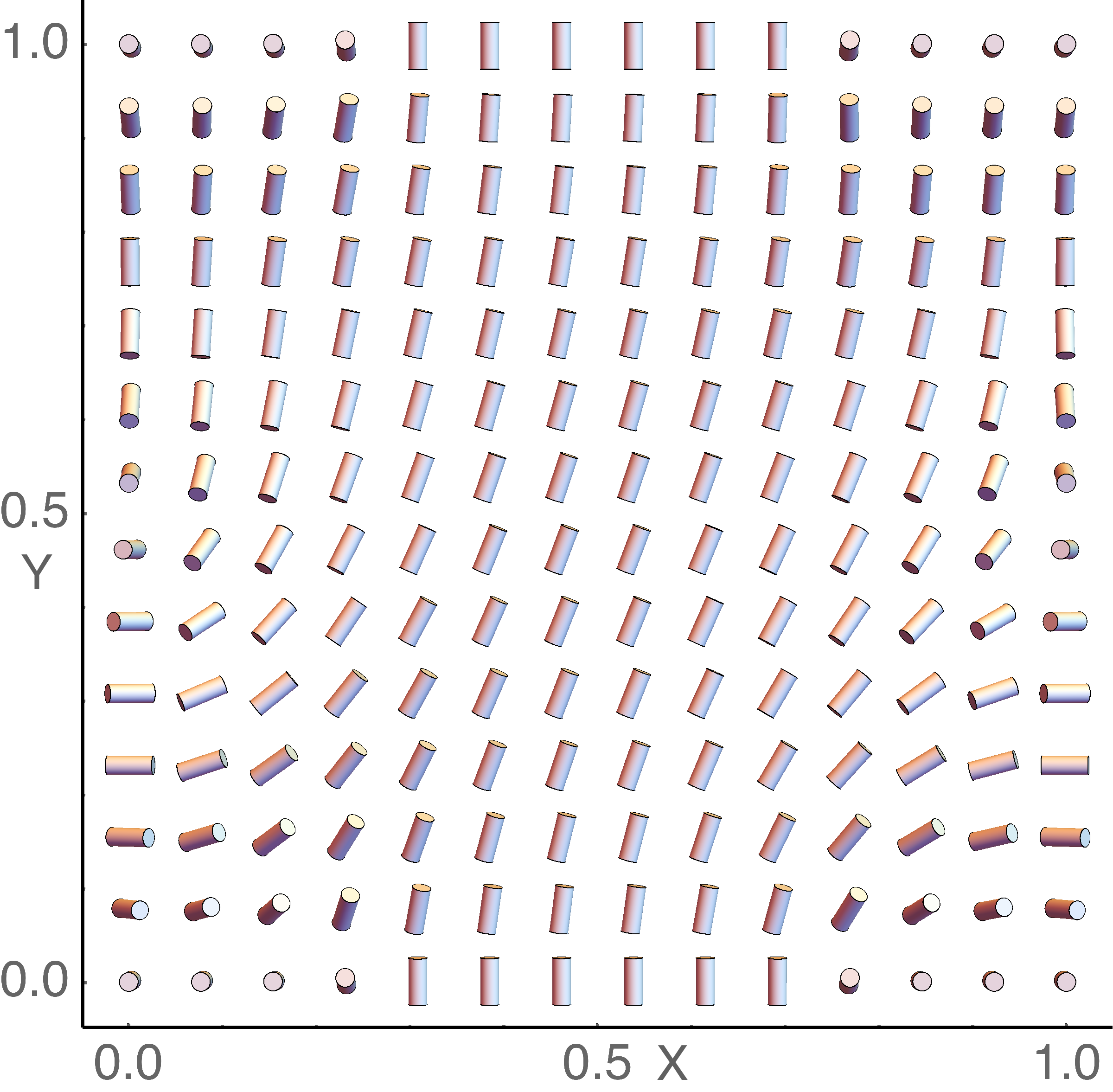}
  \label{PatternedSquareEnergyDensityPlots:left}
}
\end{minipage}
\begin{minipage}{0.49 \linewidth}
\subfloat[]{ 
  \includegraphics[scale=.113, center]{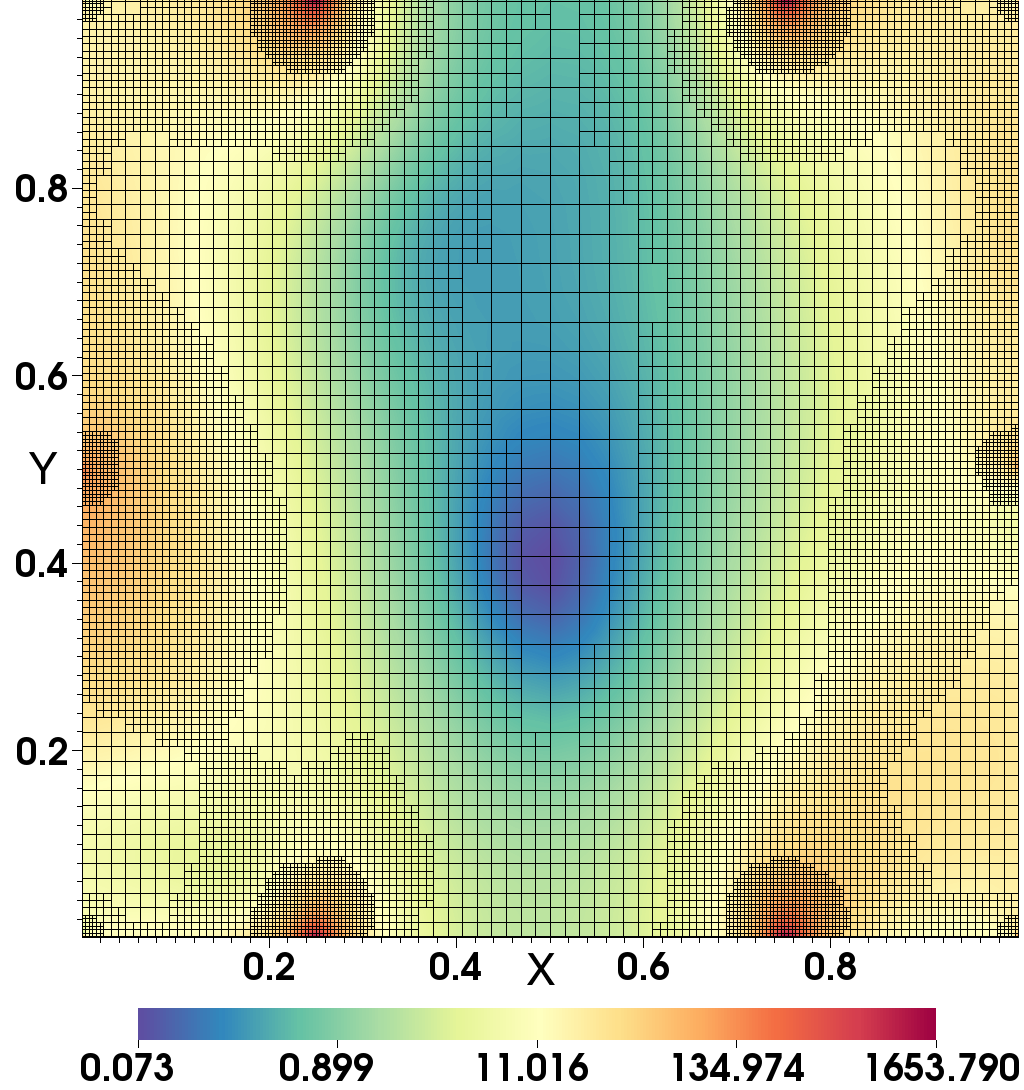}
  \label{PatternedSquareEnergyDensityPlots:center}
}
\end{minipage}
\begin{minipage}{0.49 \linewidth}
\subfloat[]{
  \includegraphics[scale=.113, center]{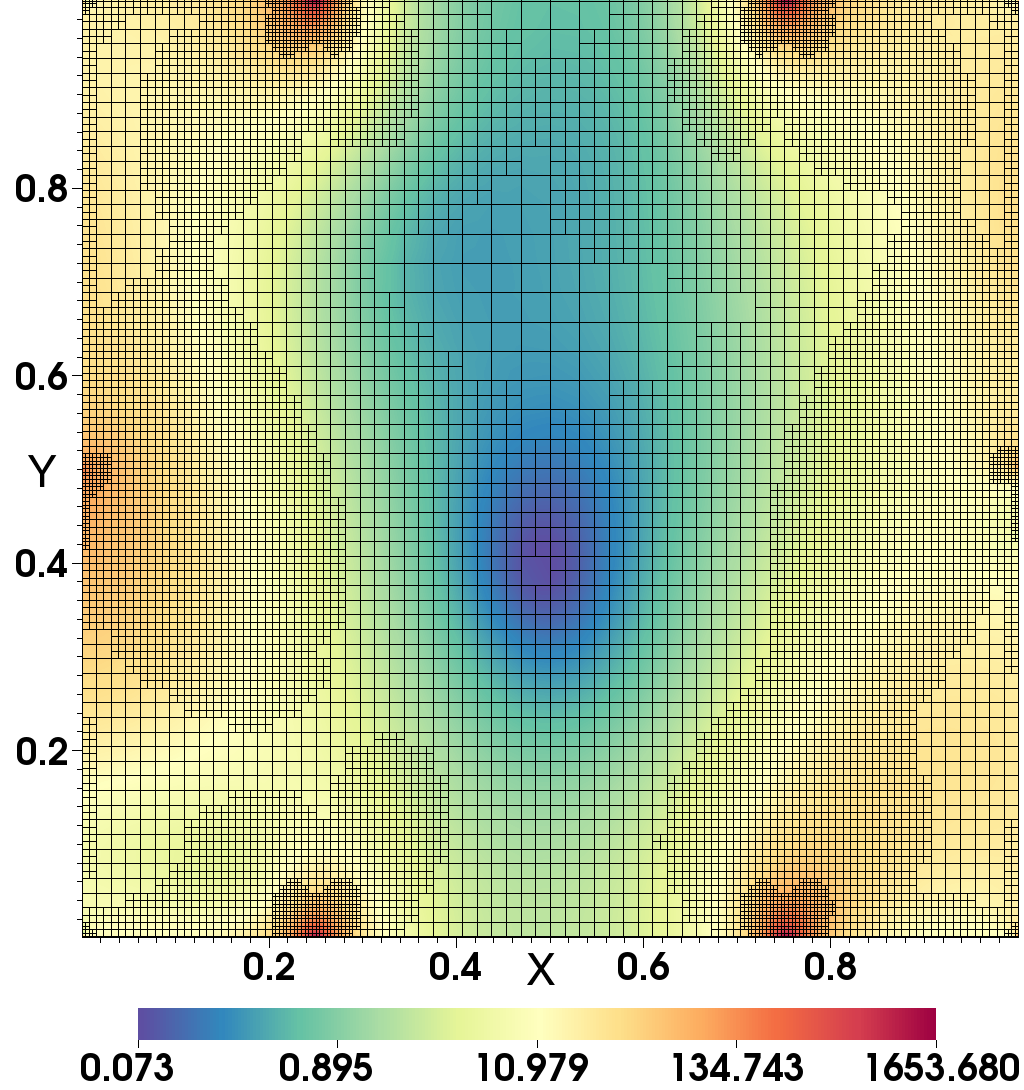}
  \label{PatternedSquareEnergyDensityPlots:right}
}
\end{minipage}
\caption{\small{\subref{PatternedSquareEnergyDensityPlots:left} The final computed solution for the Lagrangian formulation on the adapted mesh (restricted for visualization). Frank-Oseen energy-density function, $w_F$, with overlaid AMR patterns after three refinements for \subref{PatternedSquareEnergyDensityPlots:center} Lagrange multiplier approach and \subref{PatternedSquareEnergyDensityPlots:right} penalty method.}}
\label{PatternedSquareEnergyDensityPlots}
\end{figure}

The simulations in this section consider the square domain with patterned boundary conditions as shown in Figure \ref{PatternedSquareEnergyDensityPlots}\subref{PatternedSquareEnergyDensityPlots:left}. Along the boundary, the patterning induces a number of points where the director field abruptly undergoes orientational transitions leading to areas of elevated free energy. The Frank constants for these simulations are $K_1 = 1.0$, $K_2 = 0.629$, and $K_3 = 1.323$, corresponding to the non-dimensionalized constants of $5$CB \cite{Stewart1} and $t_0 = 0$. The NI hierarchy again begins on a uniform $32 \times 32$ mesh and proceeds with five additional levels of either adaptive or uniform mesh refinements. The damping parameter begins at $\alpha = 0.2$ and increases by $0.2$ after each refinement for both methods. The penalty parameter is set to $\zeta = 10^6$, where appropriate.

As in the previous experiment, the error estimators generate similar AMR patterns. Each estimator emphasizes refinement near areas of behavioral transition while leaving the central region, which contains a relatively homogeneous director field, as a set of coarser elements. Figures \ref{PatternedSquareEnergyDensityPlots}\subref{PatternedSquareEnergyDensityPlots:center} and \subref{PatternedSquareEnergyDensityPlots:right} exhibit the resulting refinement pattern after three levels of flagging and refinement. The mesh patterns display some asymmetry along both the $x$ and $y$ axes, mirroring the asymmetry of the free-energy density and director field produced by the boundary conditions.  As seen in the previous simulation, the error estimator associated with the penalty method places some additional value in refinement along the boundary. However, the difference between the two patterns in this experiment is less pronounced. 

The statistics for each method comparing AMR guided by the error estimators to a uniform refinement strategy are shown in Table \ref{PatternedExperimentStatistics} and exhibit significant efficiency improvements with adaptive refinement while maintaining nearly identical performance with regard to computed free energy and unit-length constraint conformance. Each method, regardless of refinement strategy applied, finds the same free energy for the computed equilibrium configuration. Furthermore, constraint enforcement is identical for both adaptive refinement experiments when compared to the performance of the uniformly refined meshes. For this simulation, the difference in consumed work units between the adaptive and uniformly refined mesh experiments is even higher than that in the previous section.  

\subsection{Cholesteric in Elliptic Domain}

\begin{figure*}[t]
\begin{minipage}[t]{0.33 \linewidth}
\subfloat[]{\raisebox{1.3em}{
  \includegraphics[scale=.165, center]{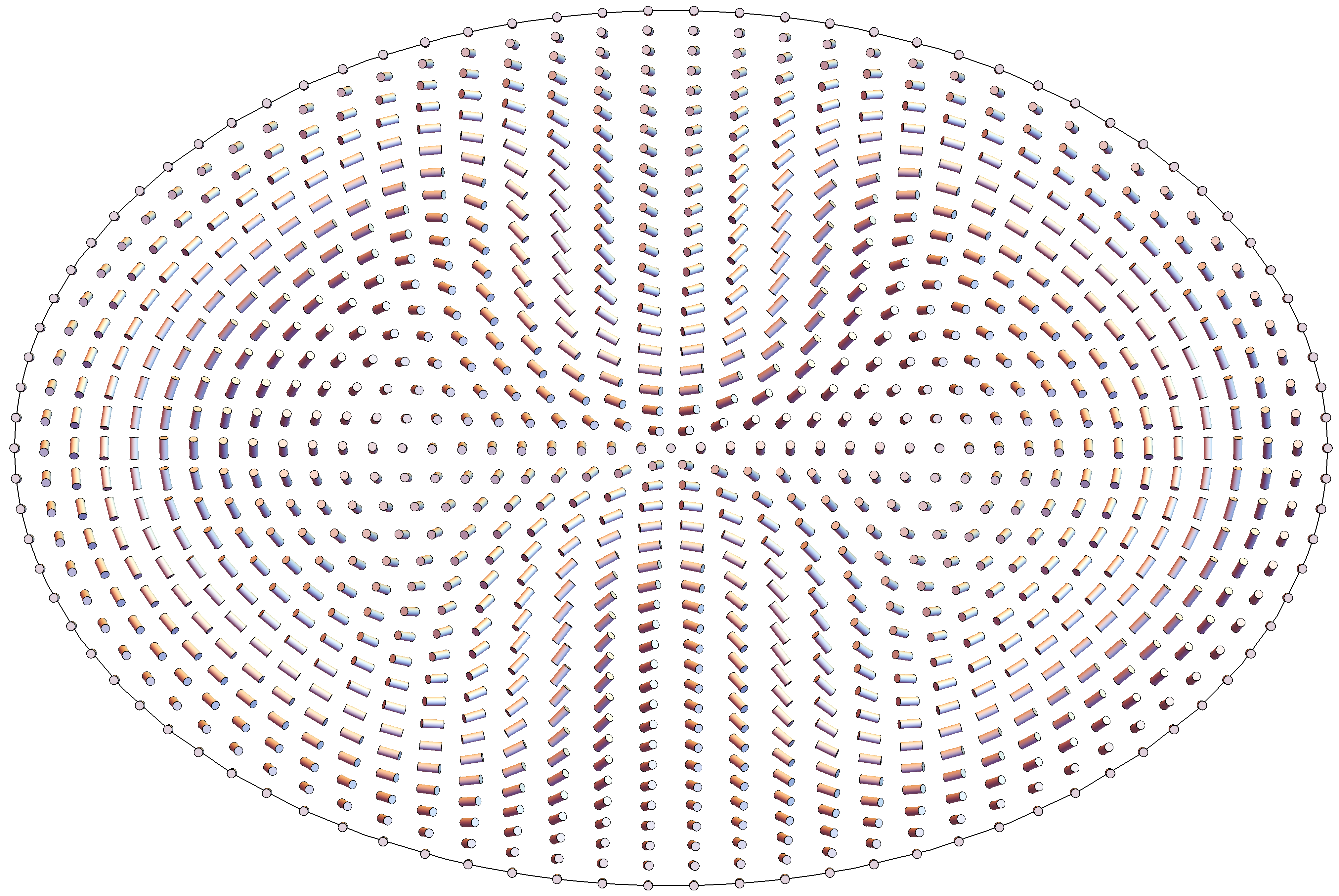}}
  \label{CholestercEllipsePlots:left1}
}
\end{minipage}
\begin{minipage}[t]{0.33 \linewidth}
\subfloat[]{
  \includegraphics[scale=.17, center]{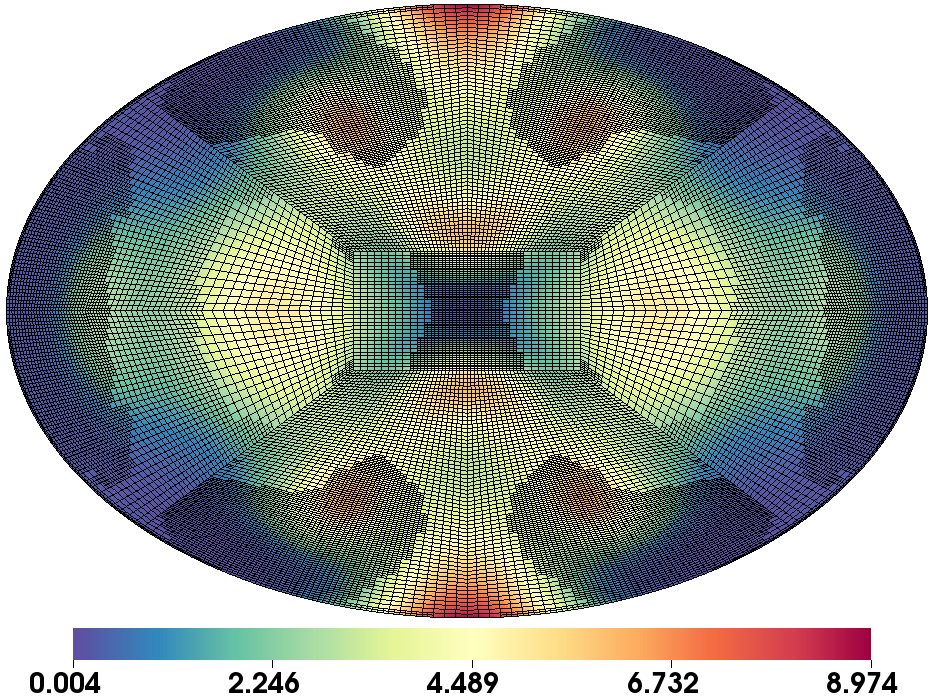}
  \label{CholestercEllipsePlots:center1}
}
\end{minipage}
\begin{minipage}[t]{0.33 \linewidth}
\subfloat[]{
  \includegraphics[scale=.17, center]{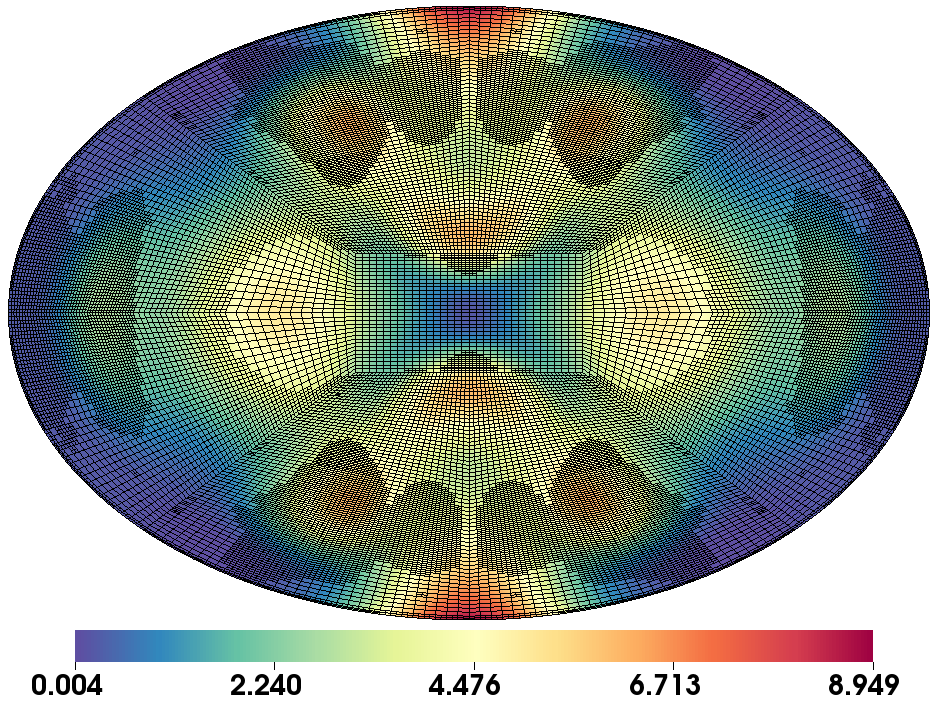}
  \label{CholestercEllipsePlots:right1}
}
\end{minipage}
\begin{minipage}[t]{0.33 \linewidth}
\subfloat[]{\raisebox{1.3em}{
  \includegraphics[scale=.165, left]{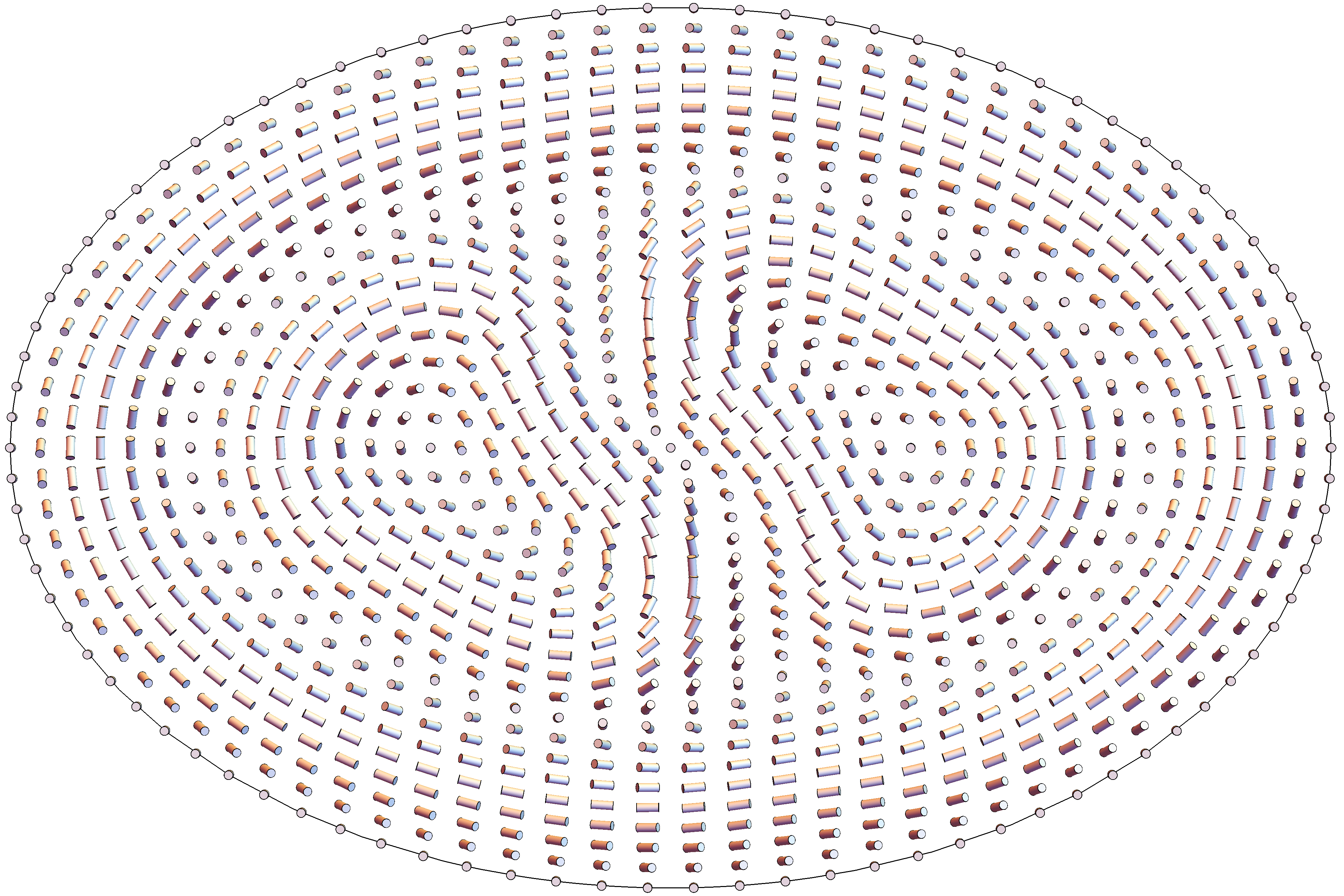}}
  \label{CholestercEllipsePlots:left2}
}
\end{minipage}
\begin{minipage}[t]{0.33 \linewidth}
\subfloat[]{
  \includegraphics[scale=.17, center]{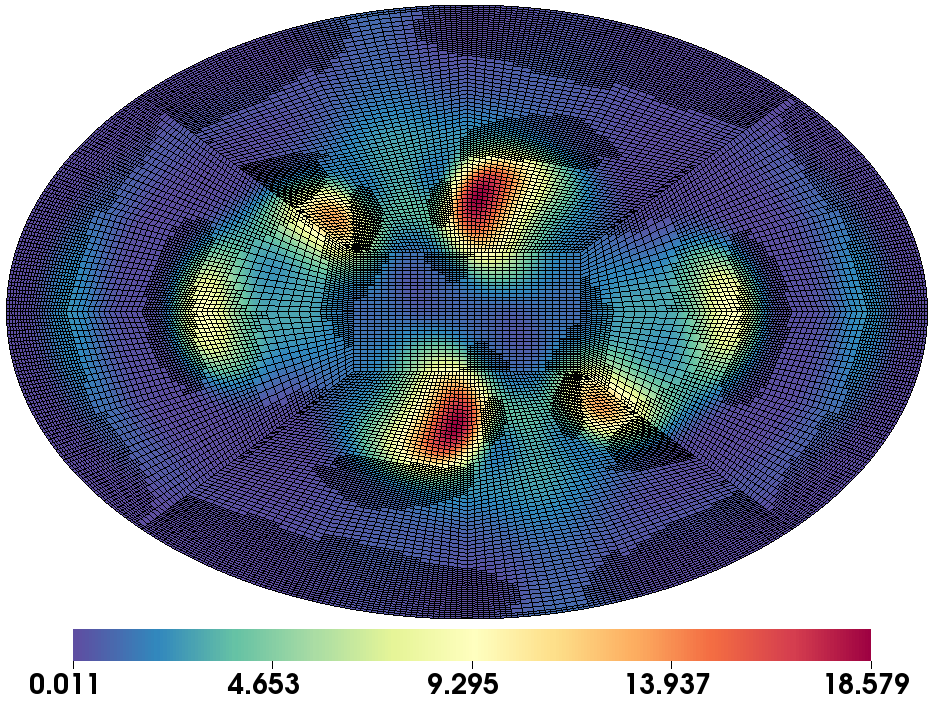}
  \label{CholestercEllipsePlots:center2}
}
\end{minipage}
\begin{minipage}[t]{0.33 \linewidth}
\subfloat[]{
  \includegraphics[scale=.17, center]{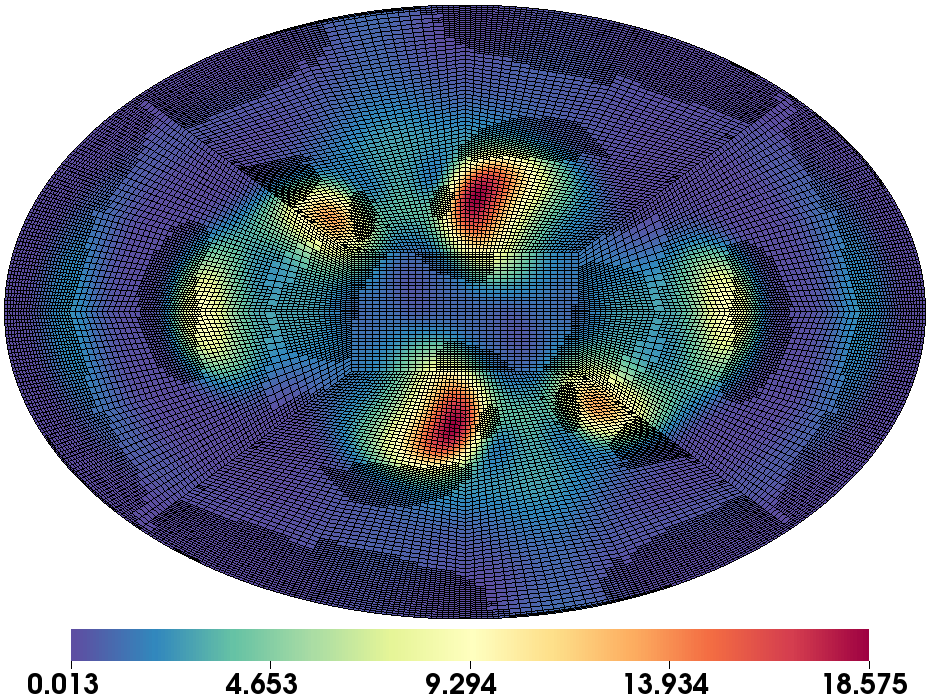}
  \label{CholestercEllipsePlots:right2}
}
\end{minipage}
\caption{\small{\subref{CholestercEllipsePlots:left1}, \subref{CholestercEllipsePlots:left2} Computed solutions for the Lagrangian formulation on the finest adapted mesh (restricted for visualization). The Frank-Oseen energy-density function, $w_F$, with overlaid AMR patterns after three refinements is shown for the Lagrange multiplier approach in \subref{CholestercEllipsePlots:center1}, \subref{CholestercEllipsePlots:center2} and the penalty method in \subref{CholestercEllipsePlots:right1}, \subref{CholestercEllipsePlots:right2}. The top row corresponds to results for $t_0 = 6.0$ while the bottom shows $t_0 = 8.0.$}}
\label{CholestercEllipsePlots}
\end{figure*}

\begin{table*}[t]
\centering
\resizebox{\textwidth}{!}{
\begin{tabular}{|c|c|c|c|c|c|c||c|c|c|c|c|c|}
\hlinewd{1.3pt}
 & \multicolumn{3}{|c|}{Penalty (Adapt.)} & \multicolumn{3}{|c||}{Penalty (Uniform)}  & \multicolumn{3}{c|}{Lagrangian (Adapt.)} & \multicolumn{3}{|c|}{Lagrangian (Uniform)} \\
\hlinewd{1.3pt}
Level & Energy & Max. Dev. & Min. Dev. & Energy & Max. Dev. & Min. Dev. & Energy & Max. Dev. & Min. Dev. & Energy & Max. Dev. & Min. Dev. \\
\hlinewd{1.3pt}
Grid $1$ & $11.986$ & $9.653$e-$04$ & $-5.880$e-$04$ & $11.986$ & $9.653$e-$04$ & $-5.880$e-$04$ & $11.990$ & $1.487$e-$03$ & $-1.500$e-$03$ & $11.990$ & $1.487$e-$03$ & $-1.500$e-$03$ \\
\hline
Grid $2$ & $11.985$ & $5.912$e-$04$ & $-2.368$e-$04$ & $11.995$ & $3.690$e-$04$ & $-6.255$e-$05$ & $12.002$ & $7.333$e-$04$ & $-6.244$e-$04$ & $12.007$ & $2.170$e-$04$ & $-2.185$e-$04$ \\
\hline
Grid $3$ & $11.998$ & $3.015$e-$04$ & $7.267$e-$06$ & $12.000$ & $2.295$e-$04$ & $1.998$e-$05$ & $12.009$ & $2.213$e-$04$ & $-1.984$e-$04$ & $12.012$ & $2.800$e-$05$ & $-2.847$e-$05$ \\
\hline
Grid $4$ & $12.001$ & $2.443$e-$04$ & $1.767$e-$05$ & $12.002$ & $2.249$e-$04$ & $1.562$e-$05$ & $12.013$ & $6.920$e-$05$ & $-6.384$e-$05$ & $12.013$ & $3.553$e-$06$ & $-3.583$e-$06$ \\
\hline
Grid $5$ & $12.002$ & $2.296$e-$04$ & $1.349$e-$05$ & $12.002$ & $2.244$e-$04$ & $1.026$e-$05$ & $12.013$ & $2.607$e-$05$ & $-2.539$e-$05$ & $12.014$ & $4.458$e-$07$ & $-4.480$e-$07$ \\
\hline
Grid $6$ & $12.002$ & $2.249$e-$04$ & $8.741$e-$06$ & $-$ & $-$ & $-$ & $12.014$ & $5.283$e-$06$ & $-4.739$e-$06$ & $-$ & $-$ & $-$ \\
\hline
\hlinewd{1.3pt}
Fine DOF & \multicolumn{3}{|c|}{$814, 575$} & \multicolumn{3}{|c||}{$3,935,235$} & \multicolumn{3}{c|}{$879,319$} & \multicolumn{3}{|c|}{$4, 263, 428$} \\
\hline
WUs & \multicolumn{3}{|c|}{$5.580$} & \multicolumn{3}{|c||}{$17.494$} & \multicolumn{3}{c|}{$4.314$} & \multicolumn{3}{|c|}{$14.819$} \\
\hline
Timing & \multicolumn{3}{|c|}{$628.4$s} & \multicolumn{3}{|c||}{$3, 158.8$s} & \multicolumn{3}{c|}{$977.4$s} & \multicolumn{3}{|c|}{$13, 047.8$s} \\
\hline
\end{tabular}
}
\caption{\small{Simulation statistics associated with the cholesteric confined in an elliptic domain with aspect ratio $1.5$ and $t_0 = 6.0$ for both adaptive and uniformly refined mesh hierarchies. The column Fine DOF reflects the number of degrees of freedom on the finest mesh of the NI. The largest director deviations above and below unit length at the quadrature nodes are shown in the Max. Dev. and Min Dev. columns.}}
\label{Ellipse1p65q6ExperimentStatistics}
\end{table*} 

The two numerical experiments of this section examine the performance of the error estimators on ellipse-type meshes. Note that the ellipses are approximated with quadrilateral meshes such that the presence of a convex polyhedral boundary, assumed in the theory, remains in place. However, after cells are flagged for refinement,  boundary elements are refined by splitting the element and placing the newly created edge node at the appropriate position on the true boundary of the ellipse to better approximate the boundary shape. Therefore, the newly created grid is not a strictly refined subdomain of the previous coarse domain. Both problems impose Dirichlet conditions such that $\director = (0, 0, 1)$ along the entire boundary. NI consists of a coarse mesh containing $1,313$ elements followed by five consecutive adaptive or uniform refinements. The elliptic domain being modeled has an aspect ratio and major axis of $1.5$. 

For these simulations, cholesteric liquid crystals with $t_0 = 6.0$ and $t_0 = 8.0$, respectively, are considered. The Frank constants are $K_1 = K_2 = K_3 = 1.0$ for the first, and $K_1 = 1.0$, $K_2 = 3.0$, and $K_3 = 1.2$ in the second. Coupling these parameters with confinement of the cholesteric in the ellipse leads to the presence of distorted equilibrium director fields on the domain interior due to geometric frustration \cite{Ackerman1, Emerson7}, offering uniquely challenging behavior on which to examine the performance of the a posteriori error estimators. For all simulations below, the Newton damping parameter begins at $\alpha = 0.3$ and increases by $0.2$ after each refinement in the NI process. The penalty parameter is set to $\zeta = 10^5$.

\begin{table*}[t]
\centering
\resizebox{\textwidth}{!}{
\begin{tabular}{|c|c|c|c|c|c|c||c|c|c|c|c|c|}
\hlinewd{1.3pt}
 & \multicolumn{3}{|c|}{Penalty (Adapt.)} & \multicolumn{3}{|c||}{Penalty (Uniform)}  & \multicolumn{3}{c|}{Lagrangian (Adapt.)} & \multicolumn{3}{|c|}{Lagrangian (Uniform)} \\
\hlinewd{1.3pt}
Level & Energy & Max. Dev. & Min. Dev. & Energy & Max. Dev. & Min. Dev. & Energy & Max. Dev. & Min. Dev. & Energy & Max. Dev. & Min. Dev. \\
\hlinewd{1.3pt}
Grid $1$ & $8.726$ & $2.219$e-$03$ & $-1.728$e-$03$ & $8.726$ & $2.219$e-$03$ & $-1.728$e-$03$ & $8.774$ & $2.931$e-$03$ & $-3.067$e-$03$ & $8.774$ & $2.931$e-$03$ & $-3.067$e-$03$ \\
\hline
Grid $2$ & $8.677$ & $1.370$e-$03$ & $-8.866$e-$04$ & $8.669$ & $7.358$e-$04$ & $-2.103$e-$04$ & $8.738$ & $3.124$e-$03$ & $-2.222$e-$03$ & $8.732$ & $4.135$e-$04$ & $-4.472$e-$04$ \\
\hline
Grid $3$ & $8.669$ & $7.357$e-$04$ & $-1.762$e-$04$ & $8.667$ & $6.586$e-$04$ & $1.376$e-$05$ & $8.730$ & $1.218$e-$03$ & $-8.633$e-$04$ & $8.730$ & $5.774$e-$05$ & $-5.709$e-$05$ \\
\hline
Grid $4$ & $8.668$ & $7.281$e-$04$ & $-4.570$e-$05$ & $8.668$ & $6.528$e-$04$ & $3.009$e-$05$ & $8.731$ & $3.564$e-$04$ & $-2.640$e-$04$ & $8.731$ & $7.503$e-$06$ & $-7.457$e-$06$ \\
\hline
Grid $5$ & $8.668$ & $6.586$e-$04$ & $3.008$e-$05$ & $8.668$ & $6.523$e-$04$ & $3.043$e-$05$ & $8.731$ & $1.430$e-$04$ & $-1.165$e-$04$ & $8.731$ & $9.534$e-$07$ & $-9.444$e-$07$ \\
\hline
Grid $6$ & $8.668$ & $6.528$e-$04$ & $2.119$e-$05$ & $-$ & $-$ & $-$ & $8.731$ & $3.555$e-$05$ & $-3.161$e-$05$ & $-$ & $-$ & $-$ \\
\hline
\hlinewd{1.3pt}
Fine DOF & \multicolumn{3}{|c|}{$810, 687$} & \multicolumn{3}{|c||}{$3,935,235$} & \multicolumn{3}{c|}{$877, 409$} & \multicolumn{3}{|c|}{$4,263,428$} \\
\hline
WUs & \multicolumn{3}{|c|}{$7.221$} & \multicolumn{3}{|c||}{$20.273$} & \multicolumn{3}{c|}{$5.845$} & \multicolumn{3}{|c|}{$19.328$} \\
\hline
Timing & \multicolumn{3}{|c|}{$735.8$s} & \multicolumn{3}{|c||}{$3,449.5$s} & \multicolumn{3}{c|}{$1,243.4$s} & \multicolumn{3}{|c|}{$15,918.3$s} \\
\hlinewd{1.3pt}
\end{tabular}
}
\caption{\small{Simulation statistics associated with the cholesteric confined in an elliptic domain with aspect ratio $1.5$ and $t_0 = 8.0$ for both adaptive and uniformly refined mesh hierarchies. The column Fine DOF reflects the number of degrees of freedom on the finest mesh of the NI. The largest director deviations above and below unit length at the quadrature nodes are shown in the Max. Dev. and Min Dev. columns.}}
\label{Ellipseq8ExperimentStatistics}
\end{table*}

Figures \ref{CholestercEllipsePlots}\subref{CholestercEllipsePlots:left1} and \subref{CholestercEllipsePlots:left2} display plots of the computed director field for the two simulations. As anticipated, confinement of the cholesterics within the elliptical boundary produces equilibrium configurations with substantial deformations on the interior of the domain. Each calculated solution exhibits challenging patterns. As can be seen in the remaining figures, some of the largest free-energy density values are located away from the domain boundary. Figures \ref{CholestercEllipsePlots}\subref{CholestercEllipsePlots:center1} and \subref{CholestercEllipsePlots:center2} display the free-energy density and AMR patterns for the Lagrange multiplier approach. Correspondingly, Figures \ref{CholestercEllipsePlots}\subref{CholestercEllipsePlots:right1} and \subref{CholestercEllipsePlots:right2} show the energy density and mesh patterns resulting from the penalty method. While refinement, in both cases, occurs along the boundary of the domain, portions of refinement also trace the regions of highest energy density and areas where the director field behavior changes dramatically. Moreover, the refined grids mirror the symmetry and some of the shape of $\director$.

Table \ref{Ellipse1p65q6ExperimentStatistics} details statistics for simulations with $t_0 = 6.0$. Results for the uniform refinement experiments are limited to four mesh levels as the memory overhead associated with constructing and solving the linearization systems for an additional refinement level is prohibitive. For both the penalty and Lagrange multiplier approaches, the computed free energies on the adapted meshes are close to those of the uniformly refined grids throughout the NI process and are in full agreement on the finest levels. Additionally, even without the final level of uniform refinement, the AMR hierarchies require less than $30\%$ of the total WUs used by the associated uniform mesh. 

Due to the complexity of the director field, enforcement of the unit-length constraint is more challenging. While still maintaining relatively strict adherence to unit length, conformance for the Lagrange multiplier approach on the finest adaptively refined grid trails the fifth uniformly refined mesh by approximately an order of magnitude. However, the constraint conformance is comparable to the solution computed on the fourth uniform mesh, which still requires $4.700$ WUs to construct compared to the $4.314$ consumed for the AMR solution. Moreover, the solution on the adapted mesh more accurately captures the free energy. With the penalty method, adherence to the constraint for the AMR scheme is quite comparable to the uniformly refined meshes. It should be noted that, by the nature of the penalty method, unit-length conformance is also dependent on the penalty parameter. For a fixed $\zeta$, a certain amount of constraint violation may result in an energetically advantageous director field. This is most likely a contributing factor to the slower rate of improvement in pointwise constraint compliance with either refinement strategy. Nevertheless, the penalty method with AMR closely tracks the performance of the uniform grids.

Data associated with the cholesteric simulations for $t_0 = 8.0$ is presented in Table \ref{Ellipseq8ExperimentStatistics}. As in the $t_0 = 6.0$ case, the constraint conformance of the computed equilibrium solution on the finest uniform mesh for the Lagrange multiplier approach is approximately an order of magnitude better than the finest adapted mesh. Nonetheless, the computed free energies on each level for the adaptively refined mesh very closely match those calculated with the uniformly refined grids. Moreover, the uniform refinement study consumes more than three times as many WUs and takes 12 times longer to finish with the current solvers. The penalty method with AMR is, again, extremely competitive with the uniform mesh hierarchy in each metric while drastically reducing computational work.

\section{Conclusions and Future Work} \label{conclusions}

In this paper, we have derived a posteriori error estimators for solutions to the nonlinear first-order optimality conditions of the Frank-Oseen model of cholesteric and nematic liquid crystal arising in the context of a penalty method and a Lagrange multiplier approach. Theory demonstrating the reliability of the error estimator for the penalty method was proven, and a discussion of current work to fully extend the results to the estimator for the Lagrange multiplier formulation was presented. In both cases, the error estimators represent readily computable quantities on each element of a finite-element mesh and, thus, a straightforward AMR strategy implemented within a nested iteration framework was proposed and investigated numerically. 

Four numerical experiments were conducted comparing the efficiency and performance of the AMR scheme based on the derived error estimators to a simple uniform refinement strategy. The adaptively refined meshes resulted in significant reductions in required computational work while maintaining favorable accuracy for the tracked statistics. In the first two simulations, solutions produced on the adaptively refined mesh are in near agreement with those found on the uniform grids for both constraint enforcement approaches but with more than a five-fold reduction in consumed WUs in each case. For the elliptical domain experiments, the Lagrange multiplier formulation with AMR generates solutions with free-energy values matching those found on uniform mesh but trails slightly in quality of pointwise unit-length enforcement. On the other hand, solutions computed with the penalty method on the adapted meshes continue to be very competitive with the uniform grids in all aspects.

As discussed above, future work will include expanding the theory presented here to demonstrate the reliability of the a posteriori error estimator proposed for the Lagrange multiplier approach, which was examined numerically herein. In addition, we aim to develop a theoretical framework to show that the derived estimators also constitute efficient estimates of approximation error for both the Lagrange multiplier and penalty formulations. Finally, studies considering an optimal choice of refinement percentage at each refinement level in the NI hierarchy will be undertaken.

\section*{\normalsize Acknowledgments}

The author would like to thank Professors James Adler, Xiaozhe Hu, and Scott MacLachlan for their extremely useful suggestions, guidance, and careful reading.



\bibliographystyle{plainnat}

\bibliography{AMRPaperArXiV.bbl}

\begin{thebibliography}{31}
\providecommand{\natexlab}[1]{#1}
\providecommand{\url}[1]{\texttt{#1}}
\expandafter\ifx\csname urlstyle\endcsname\relax
  \providecommand{\doi}[1]{doi: #1}\else
  \providecommand{\doi}{doi: \begingroup \urlstyle{rm}\Url}\fi

\bibitem[Ackerman et~al.(2014)Ackerman, Trivedi, Senyuk, van~de Lagemaat, and
  Smalyukh]{Ackerman1}
P.~J. Ackerman, R.~P. Trivedi, B.~Senyuk, J.~van~de Lagemaat, and I.~I.
  Smalyukh.
\newblock Two-dimensional skyrmions and other solitonic structures in
  confinement-frustrated chiral nematics.
\newblock \emph{Phys. Rev. E: Stat., Nonlinear, Soft Matter Phys.}, 90, 2014.

\bibitem[Adler et~al.(2015{\natexlab{a}})Adler, Atherton, Benson, Emerson, and
  Mac{L}achlan]{Emerson2}
J.~H. Adler, T.~J. Atherton, T.~R. Benson, D.~B. Emerson, and S.~P.
  Mac{L}achlan.
\newblock Energy minimization for liquid crystal equilibrium with electric and
  flexoelectric effects.
\newblock \emph{SIAM J. Sci. Comput.}, 37\penalty0 (5):\penalty0 S157--S176,
  2015{\natexlab{a}}.

\bibitem[Adler et~al.(2015{\natexlab{b}})Adler, Atherton, Emerson, and
  Mac{L}achlan]{Emerson1}
J.~H. Adler, T.~J. Atherton, D.~B. Emerson, and S.~P. Mac{L}achlan.
\newblock An energy-minimization finite-element approach for the
  {F}rank-{O}seen model of nematic liquid crystals.
\newblock \emph{SIAM J. Numer. Anal.}, 53\penalty0 (5):\penalty0 2226--2254,
  2015{\natexlab{b}}.

\bibitem[Adler et~al.(2016)Adler, Emerson, Mac{L}achlan, and
  Manteuffel]{Emerson3}
J.~H. Adler, D.~B. Emerson, S.~P. Mac{L}achlan, and T.~A. Manteuffel.
\newblock Constrained optimization for liquid crystal equilibria.
\newblock \emph{SIAM J. Sci. Comput.}, 38\penalty0 (1):\penalty0 B50--B76,
  2016.

\bibitem[Adler et~al.(2017)Adler, Emerson, Farrell, Mac{L}achlan, and
  Atherton]{Emerson7}
J.~H. Adler, D.~B. Emerson, P.~E. Farrell, S.~P. Mac{L}achlan, and T.~J.
  Atherton.
\newblock Computing equilibrium states of cholesteric liquid crystals in
  elliptical channels with deflation algorithms.
\newblock \emph{To appear: Liquid Crystals}, 2017.

\bibitem[Atherton and Adler(2012)]{Atherton1}
T.~J. Atherton and J.~H. Adler.
\newblock Competition of elasticity and flexoelectricity for bistable alignment
  of nematic liquid crystals on patterned surfaces.
\newblock \emph{Phys. Rev. E}, 86, 2012.

\bibitem[Atherton and Sambles(2006)]{Atherton2}
T.~J. Atherton and J.~R. Sambles.
\newblock Orientational transition in a nematic liquid crystal at a patterned
  surface.
\newblock \emph{Phys. Rev. E}, 74, 2006.

\bibitem[Babuska and Rheinboldt(1978)]{Babuska2}
I.~Babuska and W.~C. Rheinboldt.
\newblock A posteriori error estimates for the finite element method.
\newblock \emph{Int. J. Numer. Meth. Engng}, \penalty0 (12):\penalty0
  1597--1615, 1978.

\bibitem[Bangerth et~al.(2007)Bangerth, Hartmann, and
  Kanschat]{BangerthHartmannKanschat2007}
W.~Bangerth, R.~Hartmann, and G.~Kanschat.
\newblock {deal.II} -- a general purpose object oriented finite element
  library.
\newblock \emph{ACM Trans. Math. Softw.}, 33\penalty0 (4):\penalty0
  24/1--24/27, 2007.

\bibitem[Bank and Welfert(1985)]{Bank1}
R.~E. Bank and D.~B. Welfert.
\newblock A posteriori error estimators for elliptic partial differential
  equations.
\newblock \emph{Math. Comp.}, 44:\penalty0 283--301, 1985.

\bibitem[Brenner and Scott(1996)]{Brenner1}
S.~C. Brenner and L.~Scott.
\newblock \emph{The Mathematical Theory of Finite Element Methods}.
\newblock Springer-Verlag, 1996.

\bibitem[Cl\'{e}ment(1975)]{Clement1}
Ph. Cl\'{e}ment.
\newblock Approximation by finite element functions using local regularization.
\newblock \emph{RAIRO Anal. Num\'{e}r.}, 2:\penalty0 77--84, 1975.

\bibitem[Collings(1990)]{Collings2}
P.~J. Collings.
\newblock \emph{Liquid Crystals: Nature's Delicate Phase of Matter}.
\newblock Bristol, 1990.

\bibitem[Davis and Gartland-Jr.(1998)]{Davis1}
T.~A. Davis and E.~C. Gartland-Jr.
\newblock Finite element analysis of the {L}andau-de {G}ennes minimization
  problem for liquid crystals.
\newblock \emph{SIAM J. Numer. Anal.}, 35\penalty0 (1):\penalty0 336--362,
  1998.

\bibitem[de~Gennes and Prost(1993)]{DeGennes1}
P.~G. de~Gennes and J.~Prost.
\newblock \emph{The Physics of Liquid Crystals}.
\newblock Clarendon Press, Oxford, UK, 2nd edition, 1993.

\bibitem[Ericksen(1966)]{Ericksen2}
J.~L. Ericksen.
\newblock Inequalities in liquid crystal theory.
\newblock \emph{Phys. Fluids}, 9:\penalty0 1205--1207, 1966.

\bibitem[Frank(1958)]{Frank1}
F.~C. Frank.
\newblock On the theory of liquid crystals.
\newblock \emph{Discuss. Faraday Soc.}, 25:\penalty0 19--28, 1958.

\bibitem[John(2001)]{John4}
V.~John.
\newblock Residual a posteriori error estimates for two-level finite element
  methods for the {N}avier-{S}tokes equations.
\newblock \emph{Appl. Numer. Math.}, \penalty0 (37):\penalty0 503--518, 2001.

\bibitem[Lagerwall and Scalia(2012)]{Lagerwall1}
J.~P.~F. Lagerwall and G.~Scalia.
\newblock A new era for liquid crystal research: Applications of liquid
  crystals in soft matter, nano-, bio- and microtechnology.
\newblock \emph{Curr. Appl. Phys.}, 12\penalty0 (6):\penalty0 1387--1412, 2012.

\bibitem[Lee and Clark(2001)]{Lee1}
B.~W. Lee and N.~A. Clark.
\newblock Alignment of liquid crystals with patterned isotropic surfaces.
\newblock \emph{Science}, 291\penalty0 (5513):\penalty0 2576--2580, March 2001.

\bibitem[Leslie(1975)]{Leslie2}
F.~M. Leslie.
\newblock Distorted twisted orientation patterns in nematic liquid crystals.
\newblock \emph{Pramana, Suppl. No.}, 1:\penalty0 41--55, 1975.

\bibitem[Oden et~al.(1994)Oden, Wu, and Ainsworth]{Oden1}
J.~T. Oden, W.~Wu, and M.~Ainsworth.
\newblock An a posteriori error estimate for finite element approximations of
  the {N}avier-{S}tokes equations.
\newblock \emph{Comput. Methods Appl. Mech. Engrg}, \penalty0 (111):\penalty0
  185--202, 1994.

\bibitem[Onsager(1949)]{Onsager1}
L.~Onsager.
\newblock The effects of shape on the interaction of colloidal particles.
\newblock \emph{Ann. NY Acad. Sci.}, 51:\penalty0 627--659, 1949.

\bibitem[Starke(2000)]{Starke1}
G.~Starke.
\newblock {G}auss-{N}ewton multilevel methods for least-squares finite element
  computations of variably saturated subsurface flow.
\newblock \emph{Computing}, 64:\penalty0 323--338, 2000.

\bibitem[Stewart(2004)]{Stewart1}
I.~W. Stewart.
\newblock \emph{The Static and Dynamic Continuum Theory of Liquid Crystals: A
  Mathematical Introduction}.
\newblock Taylor and Francis, London, 2004.

\bibitem[Thomsen et~al.(2001)Thomsen, Keller, Naciri, Pink, Jeon, Shenoy, and
  Ratna]{Thomsen1}
D.~Thomsen, P.~Keller, J.~Naciri, R.~Pink, H.~Jeon, D.~Shenoy, and B.~Ratna.
\newblock Liquid crystal elastomers with mechanical properties of a muscle.
\newblock \emph{Macromolecules}, 34\penalty0 (17):\penalty0 5868--5875, 2001.

\bibitem[Verf\"{u}rth(1989)]{Verfurth3}
R.~Verf\"{u}rth.
\newblock A posteriori error estimators for the {S}tokes equations.
\newblock \emph{Numer. Math.}, 55:\penalty0 309--325, 1989.

\bibitem[Verf\"{u}rth(1994)]{Verfurth1}
R.~Verf\"{u}rth.
\newblock A posterior error estimates for nonlinear problems. {F}inite element
  discretizations of elliptic equations.
\newblock \emph{Math. Comp.}, 62\penalty0 (206):\penalty0 445--475, 1994.

\bibitem[Verf\"{u}rth(1996)]{Verfurth2}
R.~Verf\"{u}rth.
\newblock \emph{A Review of A Posteriori Error Estimation and Adaptive
  Mesh-Refinement Techniques}.
\newblock Wiley and Teubner, 1996.

\bibitem[Virga(1994)]{Virga1}
E.~G. Virga.
\newblock \emph{Variational Theories for Liquid Crystals}.
\newblock Chapman and Hall, London, 1994.

\bibitem[Yamada et~al.(2008)Yamada, Kondo, Mamiya, Yu, Kinoshita, Barrett, and
  Ikeda]{Yamada1}
M.~Yamada, M.~Kondo, J.~Mamiya, Y.~Yu, M.~Kinoshita, C.~Barrett, and T.~Ikeda.
\newblock Photomobile polymer materials: Towards light-driven plastic motors.
\newblock \emph{Angew. Chem. Int.}, 47\penalty0 (27):\penalty0 4986--4988,
  2008.

\end{thebibliography}

\end{document}